\documentclass[10pt,reqno]{amsart}
\usepackage{amsmath,amssymb,amsthm,fullpage, graphicx, chngcntr,float,xcolor,verbatim,cases,comment,tikz}
\usepackage{mathrsfs}
\usepackage{dsfont}

\newtheorem{theorem}{Theorem}[section]
\newtheorem{lemma}[theorem]{Lemma}

\newtheorem{proposition}[theorem]{Proposition}

\theoremstyle{definition}

\newtheorem{example}[theorem]{Example}

\newcommand{\ep}{\varepsilon}
\newcommand{\N}{\mathbb{N}}

\newcommand{\R}{\mathbb{R}}

\newcommand{\Z}{\mathbb{Z}}

\newcommand{\T}{\mathcal{T}}

\newcommand{\oph}{ \operatorname{Op_h} }

\newcommand{\cci}{C_c^\infty}

\newcommand{\lapm}{\Delta_{g}}
\newcommand{\wfs}{\operatorname{WF_h}}
\newcommand{\ellp}{\operatorname{ell}_h}
\newcommand{\esssupp}{\operatorname{ess-supp}}

\newcommand{\gh}{\Gamma_H}
\newcommand{\ssm}{{S^*M}}

\newcommand{\ct}{\tilde{\chi}}

\newcommand{\nua}{\nu^A}
\newcommand{\mua}{\mu^A}

\newcommand{\laa}{\lambda^A}
\newcommand{\sa}{{\Sigma^A}}

\newcommand{\hprho}{|\partial_{\bar{\xi}_1}p(\rho_0)|}

\newcommand{\xb}{\bar{x}}

\newcommand{\xib}{\bar{\xi}}

\newcommand{\xt}{\tilde{x}}
\newcommand{\xh}{\hat{x}}

\newcommand{\xit}{\tilde{\xi}}

\newcommand{\supp}{\operatorname{supp}}
\newcommand{\vol}{\operatorname{vol}}

\newcommand{\re}{\operatorname{Re}}

\counterwithin{figure}{section}

\definecolor{plum}{rgb}{1.0, 0.0, 1.0}
\definecolor{purple}{RGB}{176, 38, 255} 

\usepackage [english]{babel}
\usepackage [autostyle, english = american]{csquotes}
\MakeOuterQuote{"}

\usepackage{blindtext}

\begin{document}
\title{On the Growth of Generalized Fourier Coefficients of Restricted Eigenfunctions}

\author{Madelyne M. Brown}

\address{Department of Mathematics, University of North Carolina at Chapel Hill, Chapel Hill, North Carolina, 27599}
\email{madelyne@live.unc.edu}

\begin{abstract} Let $(M,g)$ be a smooth, compact, Riemannian manifold and $\{\phi_h\}$ a sequence of  $L^2$-normalized Laplace eigenfunctions on $M$. For a smooth submanifold $H\subset M$, we consider the growth of the restricted eigenfunctions $\phi_h|_H$ by testing them against a sequence of functions $\{\psi_h\}$ on $H$ whose wavefront set avoids $S^*H$. That is, we study what we call the generalized Fourier coefficients: $\langle \phi_h,\psi_h\rangle_{L^2(H)}$. We give an explicit bound on these coefficients depending on how the defect measures for the two collections of functions $\phi_h$ and $\psi_h$ relate. This allows us to get a little$-o$ improvement whenever the collection of recurrent directions over the wavefront set of $\psi_h$ is small.
To obtain our estimates, we utilize geodesic beam techniques.
\end{abstract}

\maketitle

\section{Introduction and Main Results}
On a smooth, compact, $n$-dimensional Riemannian manifold $(M,g)$, we consider a sequence of $L^2$-normalized Laplace eigenfunctions $\{\phi_h\}$ satisfying
\begin{equation} \label{eqn:efcn}
(-h^2 \lapm-1) \phi_h=0 \qquad \text{and} \qquad \| \phi_h\|_{L^2(M)}=1.
\end{equation}
From a quantum mechanics perspective, we can think of $\phi_h(x)$ as the wave function for a free quantum particle with fixed energy $h^{-2}$. Thus $|\phi_h(x)|^2$ gives the probability density for finding the quantum particle at $x\in M$. Understanding how these high-energy particles behave, corresponding to sending $h\to 0^+$, is a well-studied problem in mathematical physics. We are particularly interested in exploring how $\phi_h$, on average, concentrates and grows on our manifold. 

In this article, we study the generalized Fourier coefficients of $\phi_h$ restricted to a smooth, closed submanifold $H$. The Fourier expansion allows one to express $\phi_h |_H$ in terms of any complete orthonormal basis of $L^2(H)$. It is well known Laplace eigenfunctions on $H$ can be used to build such a basis of $L^2(H)$. Particularly, there exists such an orthonormal basis consisting of eigenfunctions on $H$, $\{\psi_{h_j}\}_{j\in\N}$, which satisfy
\[
-h_j^2\Delta_{g_{_H}} \psi_{h_j}=E(h_j) \psi_{h_j}  
\]
where $g_{_H}$ is the Riemannian metric on $H$ induced by $g$. Thus we can express
\begin{equation} \label{eqn:four} \phi_h|_H=\sum_{j\in\N} \langle \phi_h|_H,\psi_{h_j} \rangle_{L^2(H)} \psi_{h_j} =  \sum_{j\in\N} \left( \int_H  \phi_h \overline{\psi}_{h_j} d\sigma_H \right) \psi_{h_j} \end{equation}
where $d\sigma_H$ is the volume measure on $H$ induced by the metric $g_{_H}$. We study the Fourier coefficients in (\ref{eqn:four}), $\langle \phi_h,\psi_{h_j} \rangle_{L^2(H)}$ to gain an understanding of the restricted eigenfunctions $\phi_h|_H$.  To extract more information we instead study the growth of $|\langle \phi_h, \psi_h\rangle_{L^2(H)}|$ where $\{\psi_h\}$ is any collection of functions on $H$. We will call these the generalized Fourier coefficients. 

\subsection{Summary of Existing Results }
The growth of averages and weighted averages of eigenfunctions over a submanifold $H$ has been widely studied. Much work has been done in the case where $H$ is a smooth, closed curve, $\gamma$, and $(M,g)$ is a surface. Good \cite{good} and Hejhal \cite{Hej} showed for $\gamma$ a periodic geodesic and $(M,g)$ a  hyperbolic surface that there is a $C>0$ such that as $h\to 0^+$
\begin{equation} \label{eqn:good} \left| \int_\gamma \phi_h d\sigma_\gamma  \right| \leq C. \end{equation}
The integral in $(\ref{eqn:good})$ is typically called a period integral. Further, for $\gamma$ a unit length geodesic, Chen and Sogge  \cite{CS} showed that $\left| \int_\gamma \phi_h d\sigma_\gamma \right|\leq C \|\phi_h\|_{L^2(M)}$.  Without needing to make any global assumptions on the surface $M$ or curve $\gamma$, Xi \cite{Xi} proved for $0\leq \alpha h<c<1$ that 
\begin{equation}\label{eqn:xi}
 \left|   \int_\gamma \phi_h(\gamma(t))   e^{-i \alpha t} dt \right| \leq C |\gamma|
 \end{equation}
where $|\gamma|$ is the length of $\gamma$  

More generally, for $M$ an $n$-dimensional manifold and $H$ a submanifold of codimension $k$, Zelditch \cite{Zel} proved the sharp bound
\begin{equation}\label{eqn:stdr}
\left| \int_H \phi_h d\sigma_H  \right| =O(h^{\frac{1-k}{2}}) 
\end{equation}
which generalizes (\ref{eqn:good}). This bound has since been improved under various assumptions on $M$ and $H$ by Canzani, Galkowski, Sogge, Toth, Wyman, Xi, and Zhang \cite{CGT,IEA,CG,ECVGB,SXZ,Wy17,Wy18,Wy19a,Wy19b,Xi}. Particularly in \cite{CG}, Canzani and Galkowski show for a weight $w\in C^\infty(H)$ that
\[
 \limsup_{h\to 0^+} h^{\frac{k-1}{2}} \left| \int_H  \phi_h w \, d\sigma_H \right| \leq C_{n,k} \int_{SN^*H}  |w| \sqrt{f |H_pR_H|^{-1}} d\sigma_{SN^*H},
 \]
where $SN^*H$ is the unit conormal bundle of $H$, $H_p R_H$ is a function measuring how fast geodesics flow out of the submanifold, and $f$ is related to the defect measure of $\phi_h$.  They actually prove a stronger result for $\{\phi_h\}$ quasimodes of a wide class of semiclassical operators. To obtain their estimates, they develop a new technique that involves localizing $\phi_h$ near a family of geodesics emanating from points in $H$. Using this framework, they improve many existing results without needing global geometric conditions on their manifold. 

Under various assumptions the standard restriction bound (\ref{eqn:stdr}) has been logarithmically improved.  In \cite{SXZ}, Sogge, Xi, and Zhang study weighted period integrals on geodesics and show that there is a $C>0$ such that
\[
\left| \int_\gamma  \phi_h \, w \, d\sigma_\gamma \right| \leq C (\log (1/h))^{-1/2}, \qquad h\ll 1,
\]
for $M$ a hyperbolic surface, $\gamma$ a geodesic, and $w\in C_0^\infty$. Wyman extends this to the case where $M$ is a surface with nonpositive curvature in \cite{Wy19a} and further extends this to $k-$codimensional submanifolds in \cite{Wy18}. There he shows for manifolds with negative sectional curvature that
\begin{equation}\label{eqn:log}
\left| \int_H \phi_h d\sigma_H \right| =O\left(\frac{h^\frac{1-k}{2}}{\sqrt{\log(1/h)}} \right).
\end{equation}
In \cite{IEA}, Canzani and Galkowski give conditions on $(H,M)$ for which (\ref{eqn:log}) holds.

In this work, we allow the "weight" $w$ to be $h$-dependent. We will utilize Canzani and Galkowski's technique to obtain our results.

\subsection{Statement of Results} 
Let $H\subset M$ be a closed, embedded submanifold of codimension $k$. Let $\{\psi_h\}$ be a collection of $L^2$-normalized functions on $H$, 
\begin{equation}\label{eqn:psi}
\|\psi_h\|_{L^2(H)}=1, 
\end{equation}
and let $A=\wfs(\psi_h) \subset T^*H$ (see \cite[pg. 188]{zworski} for definition of the semiclassical wavefront set, denoted $\wfs$). We will use the coordinates $(x',\xi')$ in $T^*H$.

We assume $\{\psi_h\}$ has defect measure $\nu$  (see \cite[pg. 100]{zworski} for definition of a defect measure). Note that $\supp \nu \subset A$. Further, assume 
\begin{equation}\label{eqn:psi2}
\wfs(\psi_h) =A\Subset B^*H
\end{equation}
where $B^*H$ denotes the coball bundle in $T^*H$. Using the coordinates on $T^*H$ we can also write this as $A \Subset \{(x',\xi'):|\xi'|_{g_{_H}}<1\}$ where $g_{_H}$ is the metric induced by $g$ on $H$. We define 
\begin{equation}\label{eqn:sa}
\sa:=\{\rho\in S^*_H M:\pi_{_{T^*H}}\rho\in A\}\subseteq T^*M
\end{equation}
where $S^*_H M$ denotes the cosphere bundle with footprints in $H$ and $\pi_{_{T^*H}}$ is the projection from $T^*M$ onto $T^*H$.

We use the defect measure $\nu$ to define a measure $\nua$ on $\sa$. Essentially $\nua$ is an extension of the defect measure $\nu$ to $\sa$. We later define $\nua$ more explicitly in (\ref{eqn:nua})

In what follows we denote the recurrent set of $\sa$ by $\mathcal{R}_A$ (see Section \ref{sec:rec} for explicit definition).  Roughly, the recurrent set of $\sa$ is the collection of points $\rho\in\sa$ which, under the geodesic flow, return to $\sa$ infinitely often and eventually get arbitrarily close to the initial point $\rho$.

\begin{theorem}\label{thrm:rec}
Let $\{\phi_h\}$ be a sequence of Laplace eigenfunctions on $M$ satisfying (\ref{eqn:efcn}). Let $H\subset M$ be a closed, embedded, smooth submanifold of codimention $k$, and let $\{\psi_h\}\subset L^2(H)$ be a sequence of $L^2-$normalized functions on $H$ with defect measure $\nu$, satisfying $\wfs(\psi_h)=:A\Subset B^*H$. If $\nua(\mathcal{R}_A)=0$, where $\nua$ is defined in (\ref{eqn:nua}), then
\begin{equation}\label{eqn:thrmrec}
|\langle \phi_h, \psi_h \rangle_{L^2(H)}|=o(h^{\frac{1-k}{2}}), \qquad h\to 0^+.
\end{equation}
\end{theorem}

To the best of our knowledge, the Fourier coefficients of restricted eigenfunctions have not been studied under dynamic assumptions before. The most comparable result, \cite[Theorem 2]{CG} due to Canzani and Galkowski, gives conditions on the recurrent set of $SN^*H$ for which the period integral $\int_H \phi_h d\sigma_H$ is $o(h^\frac{1-k}{2})$ as $h\to 0^+$. If we take the collection $\psi_h\equiv 1$ we recover their result (see Example \ref{ex:yj6}). 
In Examples \ref{ex:torus2} and \ref{ex:torus} we demonstrate how Theorem \ref{thrm:rec} can be used in two different ways: to study the generalized Fourier coefficients and to understand the size of the recurrent set.

Next, instead of taking $\{\phi_h\}$ to be exact Laplace eigenfunctions, we further generalize by considering quasimodes of the form
\begin{equation}\label{eqn:quasi}
(-h^2\lapm-1) \phi_h=o_{L^2(M)}(h) \,\text{ as }\, h\to 0^+ \qquad \text{and} \qquad \| \phi_h\|_{L^2(M)}=1.
\end{equation}
We also assume $\phi_h$ is compactly microlocalized. That is, there exists a cutoff $\chi\in \cci(T^*M)$ such that
\[
(1-\oph(\chi))\phi_h = O_{C^\infty}(h^\infty).
\]
Further, let $\mu$ be a defect measure for $\phi_h$. We note that $\mu$ is supported in $S^*M$. Similar to \cite[Lemma 6 \&  Remark 3]{CGT} we use $\mu$ to define a measure on $\sa$, $\mua$, by
\begin{equation}\label{eqn:mua}
\mua(\Omega):= \lim_{T\to 0^+} \frac{1}{2T} \mu \left( \bigcup_{|t|\leq T} \varphi_t(\Omega) \right) \qquad \text{for} \, \Omega \subseteq \sa \,\, \text{ Borel}.
\end{equation}
The following theorem gives our main estimate for controlling generalized Fourier coefficients of quasimodes. Theorem \ref{thrm:rec} then follows as a corollary.

\begin{theorem}\label{thrm:6}
Let $\{\phi_h\}$ be a sequence of compactly microlocalized quasimodes on $M$ satisfying (\ref{eqn:quasi}) with defect measure $\mu$. Let $H\subset M$ be a closed, embedded, smooth submanifold of codimention $k$, and let $\{\psi_h\}\subset L^2(H)$ be a sequence of  $L^2-$normalized functions on $H$ with defect measure $\nu$, satisfying $\wfs(\psi_h)=:A\Subset B^*H$.   Further, suppose we have a Radon-Nikodym decomposition of the form
\[
\mua=f\nua+\laa
\]
where $\nua\perp \laa$ and $f\in L^1(\sa,\nua).$
Then there exists a constant $C_{n,k}>0$ depending only on $n$ and $k$ such that 
\begin{equation}\label{eqn:thrm6}
\limsup_{h\to 0^+} h^{\frac{k-1}{2}} \left| \langle \phi_h,\psi_h \rangle_{L^2(H)} \right| \leq C_{n,k} \left(  \int_{\sa}  (1-|\xi'|^2_{g_{_H}(x')})^{\frac{k-2}{2}} f  d \nua \right)^{1/2}.
\end{equation}
\end{theorem}

This gives much more explicit control on the constant in the standard restriction bound (\ref{eqn:stdr}) which gives us more insight into when  (\ref{eqn:stdr}) can be improved upon. For example, if $f=0$ in (\ref{eqn:thrm6}) then we see that we have a little-$o$ improvement. Showing that $f=0$ under the assumptions of Theorem \ref{thrm:rec} is exactly how we obtain (\ref{eqn:thrmrec}).
In the special case where $\nua$ is a volume measure on $\sa$, we can refine the proof of the theorem \label{thrm:6} to get a finer bound as follows.

\begin{theorem}\label{thrm:rem}
Let $\{\phi_h\}$ and $\{\psi_h\}$ satisfy the hypothesis of Theorem \ref{thrm:6}. Suppose $\sa\subseteq N\subseteq T^*M$ where $N$ is a smooth submanifold of dimension $d\in\N$. Further, let $m$ be the volume measure on $N$ induced from the Liouville measure on $T^*M$.  Moreover, suppose we have
\[
\mua=f\nua+\laa \qquad \text{and} \qquad \nua=u \, m
\]
where $\nua \perp \laa$, $f\in L^1(\sa,\nua)$ and $u\in C(\sa;\R)$. 
Then there exists a constant $C_{n,k,d}>0$ depending only on $n,k$, and $d$ such that 
\begin{equation}\label{eqn:nuleb}
\limsup_{h\to 0^+} h^{\frac{k-1}{2}} \left| \langle \phi_h,\psi_h \rangle_{L^2(H)} \right| \leq C_{n,k,d}  \int_{\sa} \sqrt{ (1-|\xi'|^2_{g_{_H}(x')})^{\frac{k-2}{2}} f  }|u|\, dm.
\end{equation}

\end{theorem} When we take $\{\psi_h\}$ to be an orthonormal collection of eigenfunctions on $H$ the estimate in Theorem \ref{thrm:6} allows us to study the growth of the generalized Fourier coefficients of restricted eigenfunctions. We note that the theorem holds in more generality than this, as the collection $\{\psi_h\}$ does not necessarily consist of eigenfunctions. To the best of our knowledge, the only existing results in this direction are due to Wyman, Xi, and Zelditch \cite{WXZ,wxyz}, where the authors obtain asymptotics for sums of the norm-squares of the generalized Fourier coefficients over the joint spectrum. If we take our collection $\{\psi_h\}$ independent of $h$, we recover the weighted averages result in \cite[Theorem 6]{CG}, which we demonstrate in Example \ref{ex:yj6}. We also show in Examples \ref{ex:xi1.4},  \ref{ex:xi1.3} that we are able to recover the results of \cite[Theorem 1.3]{Xi} and \cite[Theorem 1.4]{Xi}. A similar argument to \cite[Remark 1]{CG} could be used to show that we can use (\ref{eqn:nuleb}) with $H$ a single point to recover $L^\infty$ bounds of the generalized Fourier coefficients. Using such $L^\infty$ bounds, if in addition we take $\psi\equiv 1$, we could also recover the main result of \cite{STZ}.

\subsection{Examples}\label{sec:ex}

We next consider some examples to illustrate the use of Theorems \ref{thrm:rec},\ref{thrm:6}, and \ref{thrm:rem}. In the first two examples, we make use of Theorem \ref{thrm:rec} in two different ways. In the first, we show that the recurrent set has measure zero with respect to $\nua$, and hence we obtain a little-$o$ improvement. In the second example, we pick specific collections of $\phi_h$ and $\psi_h$ and explicitly compute the generalized Fourier coefficients. Then we use Theorem \ref{thrm:rec} to obtain information on the size of the recurrent set.   
In the later three examples we use Theorems \ref{thrm:6} and \ref{thrm:rem} to obtain bounds on the generalized Fourier coefficients in a few different settings. First, we take an explicit collection of $\psi_h$, second, we assume the collection of $\phi_h$'s are themselves restricted eigenfunctions, and third, when the collection of $\phi_h$ does not depend on the semiclassical parameter $h$.

We will use the coordinates $(x',\xb)$ with respect to $H$ such that $H=\{\bar{x}=0\}$ and work with dual coordinates $(\xi',\xib)$. In these coordinates we can write 
\[
\sa=\{(x',\bar{x},\xi',\bar{\xi}):|\bar{x}|=0, \, (x',\xi')\in A, \, |\xi|_g=1\}.
\]
Note that $\sa$ is parametrized by $(x',\xi',\xib)$ and that once $(x',\xi')$ are fixed, the remaining coordinate lives on the $k-1$ dimensional sphere of radius $\sqrt{1-|\xi'|_{g_{_H}}^2}$. We define the measure $\nua$ by
\begin{equation}\label{eqn:nua}
\int_{(x',0,\xi',\xib)\in \sa} f(x',\xi',\xib) d \nua(x',\xi',\xib) :=\int_{(x',\xi')\in A} \int_{\pi^{-1}(x',\xi')} \frac{f(x',\xi',\xib)}{c_k \big(1-|\xi'|_{g_{_H}(x')}^2  \big)^{\frac{k-1}{2}} } d \vol S^{k-1}_{\sqrt{1-|\xi'|_{g_{_H}}^2}}(\xib) d \nu(x',\xi')
\end{equation}
 where $c_k$ is such that $\nua(\sa)=\nu(A)=1$, $\pi$ is the projection of $\sa$ onto $A$, and $f$ is any integrable function on $\sa$.

\begin{example}[Extracting information from the dynamics]\label{ex:torus2}
Consider the torus $\mathbb{T}=\{(x,y)\in\R^2: (x,y) \sim (x+1,y) \sim (x,y+1) \}$ and a collection of $L^2$-normalized eigenfunctions  $\{\phi_h\}$ on $\mathbb{T}$. Furthermore, let $H=\{y=0\}$ and consider the collection of coherent states 
\[\psi_h=C(h)\exp\left(\frac{i}{2h}\Big(x-\frac{1}{2}\Big) \right) \exp\left(-\frac{1}{2h}\Big|x-\frac{1}{2}\Big|^2  \right)\]
on $H$ where $C(h)$ is such that $\|\psi_h\|_{L^2(H)}=1$. The wavefront set for $\{\psi_h\}$ is $A=\{(x,\xi):x=1/2,\xi=1/2\}$, and the defect measure is $\nu=\delta_{\{x=1/2,\xi=1/2\}}$. Therefore
\[
\sa=\{(x,y,\xi,\eta):x=1/2,\, y=0, \, \xi=1/2, \, \eta=\pm \sqrt{3}/{2} \}
\]
and $\nua$ is a point mass at both $(1/2,0,1/2,\sqrt{3}/2)$ and $(1/2,0,1/2,-\sqrt{3}/2)$ with mass $1/2$. Geodesics emanating from $\sa$ never return back to $\sa$ since their directions have irrational slopes. Therefore the recurrent set of $\sa$ is empty and hence $\nua(\mathcal{R}_A)=0$. Thus, Theorem \ref{thrm:rec} implies 
\[|\langle \phi_h,\psi_h \rangle_{L^2(H)}|=o(1)  \, \text{ as  } \, h\to 0^+.\]
\end{example}

\begin{example}[Obtaining information on the recurrent set]\label{ex:torus}
Consider the torus $\mathbb{T}$ and the collection of eigenfunctions on $\mathbb{T}$, $\phi_h=e^{\frac{i}{h} ( \frac{\sqrt{2}}{2} x +\frac{\sqrt{2}}{2} y)} $ where $h=\frac{\sqrt{2}}{4\pi n}$ and $n\in\N$. Furthermore, let $H=\{y=0\}$ and consider the collection of functions on $H$, $\psi_h= e^{\frac{i \sqrt{2}}{2 h} x} $. Then observe
\begin{equation}\label{eqn:ext}
|\langle \phi_h, \psi_h \rangle_{L^2(H)}|=\int_0^1 \phi_h\big|_{\{y=0\}} \overline{\psi}_h \,dx= \int_0^1 e^{\frac{i \sqrt{2}}{2h}x} e^{-\frac{i \sqrt{2}}{2 h} x} dx=1.
\end{equation}
 One can check that $A=\wfs\{\psi_h\}=\{(x,\xi):\xi=\sqrt{2}/2 \}$, $\nu=\delta_{\{\xi=\sqrt{2}/2\}} dx$,
\[
\sa=\{(x,y,\xi,\eta):y=0,\xi=\sqrt{2}/2,\eta=\pm \sqrt{2}/2 \}, 
\]
and $\nua=dx$, where we use $(\xi,\eta)$ to denote the dual coordinates to $(x,y)$. It is clear from (\ref{eqn:ext}) that $|\langle \phi_h, \psi_h \rangle_{L^2(H)}|\not=o(1)$  as $h\to 0^+$ and thus Theorem \ref{thrm:rec} implies $\nua(\mathcal{R}_A)>0$. 

For this example we can actually compute the recurrent set since the geometry is quite simple. Note that geodesics emanating from $\sa$ return to their starting point after time $n \sqrt{2}$, where $n\in\Z$.  Therefore every point of $\sa$ is recurrent and so $\nua(\mathcal{R}_A)=\nua(\sa)=1$.
\end{example}

\begin{example}[Reproducing {\cite[Theorem 1.4]{Xi}}]\label{ex:xi1.4}
Consider the simple case where we have a surface containing a smooth closed curve $\gamma$ parametrized by $t$. We consider
\[
 \limsup_{h\to 0^+}  \left| \int_\gamma \phi_h(\gamma(t)) e^{-i \alpha(h) t} dt \right| 
\]
for $\phi_h$ eigenfunctions, and some function $\alpha$ satisfying $0\leq \alpha(h) h <c<1$ and $\lim_{h\to 0^+} \alpha(h)h=\alpha_0$. We note that this is a semiclassical version of \cite[Theorem 1.4]{Xi}.  To apply our estimate, we need to normalize the exponential, thus we instead consider
\[
|\gamma|^{1/2}  \limsup_{h\to 0^+}  \left| \int_\gamma \phi_h(\gamma(t)) \frac{e^{-i \alpha(h) t}}{|\gamma|^{1/2}} dt \right| =: |\gamma|^{1/2}  \limsup_{h\to 0^+} |\langle \phi_h, \psi_h \rangle_{L^2(\gamma)}|
\]
We note that the collection $\{\psi_{h} \}=\{e^{i \alpha(h) t} |\gamma|^{-1/2}\}$ has a defect measure $\nu=|\gamma|^{-1}\delta_{\{\tau=\alpha_0\}} dt$ where $\tau$ is dual to $t$ and $dt$ denotes the Lebesgue measure on $\gamma$. Furthermore, the wavefront set $A=\wfs(\psi_{h})=\{\tau=\alpha_0\}$. Now, using $s$ to denote the coordinate on $M$ normal to $\gamma$ and $\sigma$ dual to $s$, we have
\[
\sa=\{(t,s,\tau,\sigma):s=0,\tau=\alpha_0,\sigma=\pm\sqrt{1-\alpha_0^2} \}
\]
which is one dimensional. Furthermore we compute $\nua=|\gamma|^{-1} dt$ where $dt$ is Lebesgue on $\sa$. 

Thus, applying Theorem \ref{thrm:rem}, we have that there is a $C>0$  such that
\[
|\gamma|^{1/2} \limsup_{h\to 0^+}  \left| \left\langle \phi_h, \psi_h \right\rangle_{L^2(\gamma)} \right|  \leq C |\gamma|^{1/2}  \int_{\sa} \sqrt{f (1-|\tau|^2)^{-1/2}} |\gamma|^{-1}  dt. 
\]
Next, using H\"{o}lder's inequality and that $\|f\|_{L^1(\sa,\nua)}\leq \mua(\sa)\leq1$ we obtain
\begin{align*}
|\gamma|^{1/2} \limsup_{h\to 0^+}  \left| \left\langle \phi_h, \psi_h \right\rangle_{L^2(\gamma)} \right|&\leq C |\gamma|^{1/2} \left( \int_\sa |f||\gamma|^{-1} dt \right)^{1/2} \left( \int_\sa (1-\alpha_0^2)^{-1/2} |\gamma|^{-1} dt\right)^{1/2} \\
&= C |\gamma|^{1/2} \| f \|^{1/2}_{L^1(\sa,\nua)} (1-\alpha_0^2)^{-1/4}  \leq \frac{C|\gamma|^{1/2}}{(1-\alpha_0^2)^{1/4}}. 
\end{align*}
Finally since $\alpha_0\leq c<1$ we have 
\begin{equation}\label{eqn:exxi}
 \limsup_{h\to 0^+}  \left| \int_\gamma \phi_h(\gamma(t)) e^{-i \alpha(h) t} dt \right| \leq \frac{C |\gamma|^{1/2}}{(1-\alpha_0^2)^{1/4}}  \leq  \frac{C |\gamma|^{1/2}}{(1-c^2)^{1/4}}.
\end{equation}
We see from (\ref{eqn:exxi}) that we are able to bound the Fourier coefficients by $C |\gamma|^{1/2}$ which differs from Xi's bound of $C|\gamma|$ stated in (\ref{eqn:xi}). This discrepancy is because our method uses $L^2$ norms, while Xi uses $L^1$ norms.  We also note that this example is a more general version of what Xi considered in \cite[Theorem 1.4]{Xi}, as we allow the weight $e^{-i\alpha(h) t}$ to depend on $h$.

\end{example}

%%%%%%%%%%%%%%%%%%%%%%%%%%%%%%%%%%%%%%%%%%%%%%%%%%%%%%%%%%%%%%%%%%%%%%%%%%%%%%%%%
\begin{example}[Reproducing {\cite[Theorem 1.3]{Xi}}]\label{ex:xi1.3}
As in \cite[Theorem 1.3]{Xi} we consider the case where $\phi_h$ are eigenfunctions on $M$ and $\psi_h$ are the restrictions of a eigenfunctions on $M$ to a hypersurface $H$. Let 
\[
\psi_h=\frac{\Psi_h |_{H}}{\| \Psi_h \|_{L^2(H)}} \qquad \text{where }  \Psi_h  \text{ satisfies } (-h^2 \lapm -\alpha(h)^2)\Psi_h=0 \text{ on  } M.
\]
We also assume that $0\leq \alpha(h)<c<1$ as in \cite[Theorem 1.3]{Xi} and suppose $\alpha(h)\to\alpha_0$, taking a subsequence if necessary. Since $\wfs (\Psi_{h} ) =\{|\xi|_g=\alpha_0\}$ one can see that $\wfs(\psi_{h})\subseteq \{|\xi'|_{g_H}\leq\alpha_0\}$ where we use coordinates $x=(x',\xb)$ on $M$ such that $H=\{\xb=0\}$, and dual coordinates $\xi=(\xi',\xib)$. Applying Theorem \ref{thrm:6} we have
\[
\limsup_{h \to 0^+}  \left| \langle \phi_{h}, \psi_{h} \rangle_{L^2(H)} \right|  \leq C_{n,1} \left( \int_{\sa} f (1-|\xi'|^2_{g_{_H}(x')})^{-1/2} d\nu^A \right)^{1/2}. 
\]
Furthermore, since $|\xi'|_{g_H}\leq \alpha_0$ on $\sa$, $\|f\|_{L^1(\sa,\nu^A)}\leq 1$, and $\alpha_0\leq c<1$, we obtain
\[
\limsup_{h \to 0^+}  \left| \langle \phi_{h}, \psi_{h} \rangle_{L^2(H)} \right|    \leq   C_{n,1}  \| f \|^{1/2}_{L^1(\sa,\nua)}   (1-\alpha_0^2)^{-1/4} \leq \frac{C_{n,1} }{(1-c^2)^{1/4}}.
\]
Thus, we find that for $h$ small 
\[
|\langle \phi_{h},\Psi_{h} \rangle_{L^2(H)} | \leq \frac{ C_{n,1}  \| \Psi_{h} \|_{L^2(H)}}{ (1-c^2)^{1/4}} \leq \frac{ C \left(1+ \frac{\alpha(h)}{h} \right)^{1/4}}{(1-c^2)^{1/4}} 
\]
where we use \cite[Theorem 3]{BGT} to bound $\| \Psi_{h} \|_{L^2(H)}$. In this case we recover the bound in \cite[Theorem 1.3 (1.23)]{Xi}. 
\end{example}

%%%%%%%%%%%%%%%%%%%%%%%%%%%%%%%%%%%%%%%%%%%%%%%%%%%%%%%%%%%%%%%%%%%%%%%%%%%%%%%%%
\begin{example}[Reproducing {\cite[Theorem 6]{CG}}]\label{ex:yj6}
We study the case where our collection $\{\psi_h\}$ does not depend on $h$. We consider 
\[
\limsup_{h\to 0^+} h^\frac{k-1}{2} | \langle \phi_h, w \rangle_{L^2(H)} | 
\]
where $\phi_h$ are compactly microlocalized quasimodes, and $w\in C^\infty(H)$ is independent of $h$. We must normalize $w$ to apply the theorem. We instead consider $\mathbf{w}=w \|w\|_{L^2(H)}^{-1}$. A short calculation shows that $\nu=\|w\|_{L^2(H)}^{-2} |w(x')|^2 \delta_{\{\xi'=0\}} dx'$ is the defect measure for $\mathbf{w}$ where we use coordinates $x=(x',\xb)$ on $M$ such that $H=\{\xb=0\}$, and dual coordinates $\xi=(\xi',\xib)$. Furthermore we observe that $A=\wfs(\mathbf{w})=N^*H$. Therefore $\sa= SN^*H$, which is $n-1$ dimensional. Next we note
\[
\nua= \|w\|_{L^2(H)}^{-2} |w(x')|^2 d\sigma_{SN^*H}
\]
where $\sigma_{SN^*H}$ is the measure on $SN^*H$ induced by the Sasaki metric on $T^*M$. Applying Theorem \ref{thrm:rem} we have
\begin{align*}
\limsup_{h\to 0^+} h^\frac{k-1}{2} | \langle \phi_h, \mathbf{w} \rangle_{L^2(H)} |   &\leq  \frac{C_{n,k,n-1}}{ \|w\|^2_{L^2(H)}}  \int_{SN^*H}  \sqrt{f (1-|\xi'|_{g_{_H}(x')}^2)^{\frac{k-2}{2} } }|w|^2  d\sigma_{SN^*H} \\
&=  \frac{C_{n,k}}{ \|w\|^2_{L^2(H)}}  \int_{SN^*H}  \sqrt{f }|w|^2  d\sigma_{SN^*H}
\end{align*}
since $\xi'=0$ on $SN^*H$. Note that in the notation of Theorem \ref{thrm:rem} we have $u=\|w\|^2_{L^2(H)}|w|^2$. In addition, since the dimension of $\sa$ is $n-1$ we just have that our constant depends on $n$ and $k$.
Thus, for the inner product with $w$, we have
\[
\limsup_{h\to 0^+} h^\frac{k-1}{2} | \langle \phi_h, w \rangle_{L^2(H)} |  \leq \frac{C_{n,k}}{\|w\|_{L^2(H)}} \int_{SN^*H} \sqrt{f} |w|^2  d\sigma_{SN^*H}=C_{n,k} \int_{SN^*H} \sqrt{f \|w\|^{-2}_{L^2(H)} |w|^2} |w| d\sigma_{SN^*H}  .
\]
The last equality matches with the bound in \cite[Theorem 6]{CG}, since under the square root is the Radon-Nikodym derivative of $\mua$ with respect to $\sigma_{SN^*H}$, which in this case is $fu=f\|w\|^{-2}_{L^2(H)}|w|^2$. 
\end{example}

\subsection{Organization of the paper} The remaining sections of our paper are organized as follows: Section 2 contains the proofs of Theorems \ref{thrm:6} and \ref{thrm:rem} assuming a key quantitative estimate given in Proposition \ref{prop}.  Section 3 contains a few of the more technical lemmas, which focus on localizing to $\sa$, needed to prove Proposition \ref{prop}. Section 3 can be omitted on a first read. Section 4 is dedicated to the proof of Proposition \ref{prop} in which the key idea is to first localize the generalized Fourier coefficients to geodesic tubes emanating from $\sa$. In Section 5 we define the recurrent set of $\sa$ and use Theorem \ref{thrm:6} to prove Theorem \ref{thrm:rec}.

\section{Proof of Theorem \ref{thrm:6} and Theorem \ref{thrm:rem}}
In this section we present the proofs of Theorems \ref{thrm:6} and \ref{thrm:rem}. We first introduce notation that will be used throughout the paper. Then we state the main estimate, Proposition \ref{prop}, which is central to the proof of Theorem \ref{thrm:6}, but we save its proof for Section \ref{sec:proofs}. Assuming the proposition, we prove Theorem \ref{thrm:6} and then modify its proof to obtain Theorem \ref{thrm:rem}.

Throughout this section we assume $\{\phi_h\}$ is a compactly microlocalized collection of quasimodes on $M$ satisfying (\ref{eqn:quasi}) with defect measure $\mu$. We also assume that the sequence of functions $\{\psi_h\}$  on $H$ have defect measure $\nu$ and satisfy (\ref{eqn:psi}) and (\ref{eqn:psi2}).

\subsection*{Acknowledgements}
The author would like to thank Yaiza Canzani and Jeffrey Galkowski for many insightful conversations throughout the course of this project and for their feedback on multiple drafts of the article. The author would also like to thank Blake Keeler for the many helpful discussions, especially at the early stages of this project. The author is grateful to the National Science Foundation for their support through the Graduate Research Fellowship Program (DGE-1650116).

\subsection{Preliminaries}
We let $P(h):=-h^2\lapm-1$ with principal symbol $p(x,\xi)=|\xi|_g^2-1$. Then we can rewrite the quasimode equation for $\phi_h$ as, $P(h) \phi_h =o_{L^2(M)}(h)$. 
Using properties of defect measures, we know that
\[\supp \mu \subseteq \{p=0\}= \{|\xi|_g^2=1\}=S^*M\subseteq T^*M,\]
so $\{\phi_h\}$ is localized near $S^*M$. Also, since $A=\wfs(\psi_h)$, we note that $\sa$, defined in (\ref{eqn:sa}), can be thought of points where $\phi_h$ are concentrated which project onto where $\psi_h$ are concentrated. Therefore, it is reasonable to expect contributions from $|\langle \phi_h,\psi_h \rangle_{L^2(H)}|$ to be small away from $\sa$. We prove this  in Lemma \ref{lem:ohinf}. 

We use $H_p$ to denote the Hamilton vector field associated to $p$ and $\varphi_t:=\exp(t H_p)$ to denote the geodesic flow. Let $\mathscr{L}\subset T^*M$ be a smooth, embedded hypersurface containing $\sa$ which is transversal to the flow, so 
\[H_p \not\in T\mathscr{L} \qquad \text{and} \qquad \sa\subset \mathscr{L}\] 
as depicted in Figure \ref{fig:tub}. For $\rho \in\mathscr{L}$ and $R>0$ define
\[
B_{\mathscr{L}}(\rho,R):=B(\rho,R)\cap \mathscr{L}.
\]
We use the geodesic flow to form tubes in $T^*M$ by flowing out of $\mathscr{L}$. For time $T>0$  and $U\subset \mathscr{L}$ we define the tube
\begin{equation}\label{eqn:tubes}
\T^T(U):=\bigcup_{|t|\leq T} \varphi_t(U).
\end{equation}
Sometimes when $U$ is a ball, we will write $\T^T_{\rho_0}(R):=\T^T(B_\mathscr{L}(\rho_0,R))$. 
For $U\subseteq \sa$  and $\ep>0$ we define 
\begin{equation}\label{eqn:thicc}
U(\ep):=\bigcup_{\rho\in U} B_{\mathscr{L}}(\rho,\ep) \subseteq \mathscr{L}
\end{equation}
which is a version of $U$ that has been thickened by $\ep$ into $\mathscr{L}$. We denote the "flowout" of $\sa$
\[
{\Lambda^T}(\ep):=\T^T(\sa(\ep))
\]
where $\sa(\ep)$ denotes the fattened version of $\sa$ defined in (\ref{eqn:thicc}).
Finally, define $\gh:C^\infty(M)\to C^\infty(H)$ which restricts functions on $M$ to $H$.
 
 \begin{figure}
\begin{center}
	\includegraphics[scale=0.5]{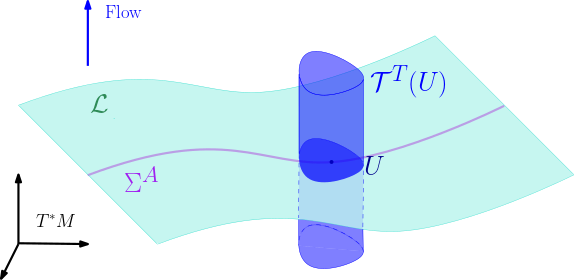}
	\end{center}
	\caption{}\label{fig:tub}
\end{figure} 

%%%%%%%%%%%%%%%%%%%%%%%%%%%%%%%%%%%%%%%%%%%%%%%%%%%%%%%%%%%%%%%%%%%%%%%%%%%%%%%%%
 To prove Theorem \ref{thrm:6} we begin by using a cutoff $\chi$ to localize to the respective supports of our mutually singular measures, $\nua$ and $\laa$. Thus we seek to understand how terms like $ \left| \langle \oph(\chi) \phi_h,\psi_h \rangle_{L^2(H)} \right| $ grow as $h\to 0^+$. We control such terms in the following proposition.
 
%%%%%%%%%%%%%%%%%%%%%%%%%%%%%%%%%%%%%%%%%%%%%%%%%%%%%%%%%%%%%%%%%%%%%%%%%%%%%%%%%
\begin{proposition} \label{prop} There exist $T_0,R_0>0$ such that for all $0<T\leq T_0$, $0<\ep\leq R_0$, and $\chi\in\cci(T^*M)$ with $H_p\chi\equiv 0$ on $\Lambda^{2T}(\ep)$, there exists a constant $C_{n,k}>0$ depending only on $n$ and $k$ such that

\[
\limsup_{h\to 0^+} h^{\frac{k-1}{2}} \left| \langle  \oph(\chi) \phi_h,  \psi_h  \rangle_{L^2(H)}  \right|  \leq C_{n,k} \left( \nua(\supp\chi|_{_{\sa}}) \int_{\sa} (1-|\xi'|^2_{g_{_H}(x')})^\frac{k-2}{2}   |\chi|^2 d\mua \right)^{1/2}.
\]
\end{proposition}
%%%%%%%%%%%%%%%%%%%%%%%%%%%%%%%%%%%%%%%%%%%%%%%%%%%%%%%%%%%%%%%%%%%%%%%%%%%%

To use Proposition \ref{prop}, we need to work with cutoff functions $\chi\in\cci(T^*M)$ in which are flow invariant, meaning $H_p\chi\equiv 0$ on $\Lambda^{2T}(\ep)$. In the following lemma, we show that a cutoff $\tilde{\chi}\in\cci(\sa(\ep))$ can be extended to $T^*M$ in this way. 

\begin{lemma}\label{lem:flow}
For $\ep>0$ and $\tilde\chi\in\cci({\sa}(\ep);[0,1])$ there exists an extension $\chi\in\cci(T^*M,[0,1])$  such that $\supp \chi \subseteq {\Lambda^{3T}}(\ep)$ and $H_p \chi\equiv 0$ on $\Lambda^{2T}(\ep)$.
\end{lemma}
\begin{proof}

Since ${\sa}(\ep) \subseteq \mathscr{L}$ is transverse to the flow, for $T$ small enough, we can use the map $X: (-4T,4T)\times\mathscr{L} \to T^*M$ defined by
\[
X(t,\rho)=\varphi_t(\rho) 
\]
as coordinates. Let $f\in \cci(\R)$ with $\supp f\subset(-3T,3T)$ and  $f\equiv 1$ on $[-2T,2T]$. Then take $\chi=(X^{-1})^*(f(t)\tilde{\chi})$.
\end{proof}

We first prove Theorem \ref{thrm:6} assuming the proposition holds. The proof of Proposition \ref{prop} is saved for Section \ref{sec:proofs}.

%%%%%%%%%%%%%%%%%%%%%%%%%%%%%%%%%%%%%%%%%%%%%%%%%%%%%%%%%%%%%%%%%%%%%%%%%%%%%%%%%
%%%THEOREM%%%%%%%%%%%%%%%%%%%%%%%%%%%%%%%%%%%%%%%%%%%%%%%%%%%%%%%%%%%%%%%%%
\subsection{Proof of Theorem \ref{thrm:6}}\label{sec:pfth6}
Fix, $\delta>0$. Since $\nua$ and $\laa$ are mutually singular Radon measures on $\sa$ there exists $K_\delta \subseteq \sa$ compact and $U_\delta \subseteq \sa$ open and containing $K_\delta$ such that
\[
\nua(U_\delta)\leq \delta \qquad \text{and}\qquad \laa(\sa\setminus K_\delta) \leq \delta.
\]
Let $\tilde{\kappa}_\delta \in \cci(\sa;[0,1])$ such that
\[
\tilde{\kappa}_\delta \equiv 1 \text{ on } K_\delta \qquad \text{and} \qquad \supp \tilde{\kappa}_\delta \subseteq U_\delta.
\]
Furthermore, let $\kappa_\delta\in \cci(T^*M;[0,1])$ be a flow invariant extension of $\tilde{\kappa}_\delta$ as defined in Lemma \ref{lem:flow}. We split the inner product 
\begin{equation}\label{eqn:2terms}
\limsup_{h\to 0^+} h^{\frac{k-1}{2}} \left| \langle \phi_h, \psi_h \rangle \right| \leq \limsup_{h\to 0^+} h^{\frac{k-1}{2}}   \Big( \left|  \langle \oph(\kappa_\delta) \phi_h, \psi_h  \rangle_{L^2(H)}  \right| +   \left|  \langle \oph(1-\kappa_\delta) \phi_h, \psi_h  \rangle_{L^2(H)}  \right| \Big).
\end{equation}
Next, we use Proposition \ref{prop} with $\chi=\kappa_\delta$ on the first term to obtain 
\begin{equation}\label{eqn:kdelta1}
\limsup_{h\to 0^+} h^{\frac{k-1}{2}}   \left|  \langle \oph(\kappa_\delta) \phi_h, \psi_h  \rangle_{L^2(H)} \right|  \leq C_{n,k} \left( \nua(\supp \kappa_\delta|_{_\sa}) \int_{\sa}  (1-|\xi'|^2_{g_{_H}(x')})^\frac{k-2}{2}   |\kappa_\delta|^2 d\mua \right)^{1/2} \leq C\delta^{1/2}.
\end{equation}
The last inequality follows from the fact that $\nua(\supp \kappa_\delta|_{_\sa})=\nua(\supp \tilde{\kappa_\delta})\leq \nua(U_\delta)\leq \delta.$

Next, to bound the second term in (\ref{eqn:2terms}), we use Proposition \ref{prop} with $\chi=1-\kappa_\delta$ and the Radon-Nikodym decomposition of our measures, $\mua=f \nua+\lambda^A$. We have
\begin{align}
\limsup_{h\to 0^+} h^{\frac{k-1}{2}}  &  \left|  \langle \oph(1-\kappa_\delta) \phi_h, \psi_h  \rangle_{L^2(H)}  \right|  \nonumber \\
&\leq C_{n,k}  \nua(\supp (1-\kappa_\delta) \big|_{_\sa})^{1/2} \left( \int_{\sa} (1-|\xi'|^2_{g_{_H}(x')})^\frac{k-2}{2} |1-\kappa_\delta|^2 \big( f d\nua + d\laa \big) \right)^{1/2} \nonumber  \\
& \leq C_{n,k}  \nua(\sa)^{1/2} \left( \int_{\sa} (1-|\xi'|^2_{g_{_H}(x')})^\frac{k-2}{2}  f d\nua +  C\delta \right)^{1/2}, \label{eqn:kdelta}
\end{align}
where, in the last line, we used that $\tilde{\kappa}_\delta \equiv 1$ on $K_\delta$ and so $(1-\kappa_\delta)\big|_{_\sa}$ is supported on $\sa\setminus K_\delta$. Thus, since $\laa(\sa\setminus K_\delta)\leq \delta$, the $d\laa$ integral is bounded by  $C \delta$. 
Since $\nua(\sa)=1$, and (\ref{eqn:kdelta1}) and (\ref{eqn:kdelta}) hold for all $\delta>0$, combining the above we have
\[
\limsup_{h\to 0^+} h^{\frac{k-1}{2}} \left| \langle \phi_h, \psi_h \rangle \right| \leq C_{n,k}  \left( \int_{\sa} (1-|\xi'|^2_{g_{_H}(x')})^\frac{k-2}{2}  f d\nua \right)^{1/2} 
\]
giving the bound in (\ref{eqn:thrm6}) as desired.
$\hfill\square$

%%%%%%%%%%%%%%%%%%%%%%%%%%%%%%%%%%%%%%%%%%%%%%%%%%%%%%%%%%%%%%%%%%%%%%%%%%%%%%%%%
\subsection{Proof of Theorem \ref{thrm:rem}} Let $K_\delta, U_\delta,$ and $\kappa_\delta$ be as in the proof of Theorem \ref{thrm:6}. We similarly split the inner product:
\begin{equation*}
\limsup_{h\to 0^+} h^{\frac{k-1}{2}} \left| \langle \phi_h, \psi_h \rangle \right|  \leq \limsup_{h\to 0^+} h^{\frac{k-1}{2}}   \Big( \left|  \langle \oph(\kappa_\delta) \phi_h, \psi_h  \rangle_{L^2(H)}  \right| +   \left|  \langle \oph(1-\kappa_\delta) \phi_h, \psi_h  \rangle_{L^2(H)}  \right| \Big) =:I+I\!I. \\
\end{equation*}
Then applying Proposition \ref{prop} to $I$, we have
\[
\limsup_{h\to 0^+} h^{\frac{k-1}{2}} \left| \langle \phi_h, \psi_h \rangle \right| \leq C\delta^{1/2} + I\!I.
\]
By the Besicovitch-Federer Covering Lemma, there exists a constant $c_d>0$ depending only on $d$, the dimension of $\sa$  and $R$ so that for all $0<r<R$, there exist a cover of open balls $\{ B(\rho_1,r),\dots, B(\rho_{n(r)},r) \}=\{B_1,\dots, B_{n(r)}\}\subseteq \sa$ of radius $r$ centered at $\{\rho_1,\dots,\rho_{n(r)}\}$ with
\[
n(r)\leq c_d  r^{-d} \qquad \text{and} \qquad m(B_j) \leq c_d r^d
\]
 where $m$ is Lebesgue on $N\supseteq \sa$. Furthermore $\sa \subseteq \bigcup_{j=1}^{n(r)} B_j$ and each point in $\sa$ lies in at most $c_d$ balls. Then we let $\tilde{\theta}_j$ be a partition of unity associated to $B_j(\ep)$ and $\theta_j$ the flowed extensions into $T^*M$ such that $\supp H_p \theta_j \subseteq \T^{3T}(B_j(\ep)) \setminus \T^{2T}(B_j(\ep))$ and $\sum_{j=1}^{n(r)} \theta_j \equiv 1$ on $\Lambda^{2T}(\ep)$. Define $\Theta:=\sum_{j=1}^{n(r)} \theta_j$. Next we split $I\!I$:
\begin{align}
\limsup_{h\to 0^+} h^\frac{k-1}{2}|  \langle &\oph(1-\kappa_\delta) \phi_h,  \psi_h  \rangle_{L^2(H)}  | \nonumber \\
& \leq \limsup_{h\to 0^+}  h^\frac{k-1}{2} \Big( \left|  \langle \oph(\Theta(1-\kappa_\delta)) \phi_h, \psi_h  \rangle_{L^2(H)}  \right| +  \left|  \langle \oph((1-\Theta)(1-\kappa_\delta)) \phi_h, \psi_h  \rangle_{L^2(H)}  \right| \Big) \nonumber \\
& \leq   \limsup_{h\to 0^+} \sum_{j=1}^{n(r)}  h^\frac{k-1}{2} \left|  \langle \oph(\theta_j(1-\kappa_\delta)) \phi_h, \psi_h  \rangle_{L^2(H)}  \right|  \nonumber \\
& \qquad+  \limsup_{h\to 0^+}h^\frac{k-1}{2} \left|  \langle \oph((1-\Theta)(1-\kappa_\delta)) \phi_h, \psi_h  \rangle_{L^2(H)}  \right|.\label{eqn:splitagain}
\end{align}
Taking $h\to 0^+$ we can apply Proposition \ref{prop} to both terms. Using the support properties of $\Theta$, we find that the second term in (\ref{eqn:splitagain}) goes to $0$. For the first term in (\ref{eqn:splitagain}), we have
\begin{align}
\limsup_{h\to 0^+} \sum_{j=1}^{n(r)} & h^\frac{k-1}{2} \left|  \langle \oph(\theta_j(1-\kappa_\delta)) \phi_h, \psi_h  \rangle_{L^2(H)}  \right| \nonumber \\
&\leq C_{n,k} \sum_{j=1}^{n(r)} \nua( \supp (\theta_j (1-\kappa_\delta))|_{_\sa})^{1/2} \left(\int_\sa (1-|\xi'|^2_{g_{_H}(x')})^{\frac{k-2}{2}} | \theta_j (1-\kappa_\delta)|^2  d\mua  \right)^{1/2} \nonumber \\
& \leq C_{n,k}  \sum_{j=1}^{n(r)} \left( \int_{B_j} u \, dm  \right)^{1/2} \left(\int_\sa (1-|\xi'|^2_{g_{_H}(x')})^{\frac{k-2}{2}} | \theta_j (1-\kappa_\delta)|^2  (f u\, dm+ d\laa)   \right)^{1/2},  \label{eqn:rem} 
\end{align}
where we used that $\supp \theta_j|_{_\sa}\subseteq B_j$ and $\nua=u\, m$. As in the proof of Theorem \ref{thrm:6}, the $d\laa$ integral can be bounded by $C\delta$, and we thus focus on the $dm$ integral. Since $u$ is uniformly continuous on $\sa$, we can find an $R>0$ such that if $\rho\in B(\rho_j,R)$ then $|u(\rho)-u(\rho_j)|\leq\delta$. Therefore, 
\[
\int_{B_j} u \, dm \leq \int_{B_j}( u(\rho_j)+\delta) \, dm =( u(\rho_j)+\delta) m(B_j) \leq c_d r^d ( u(\rho_j)+\delta).
\] 
For each $B_j$ provided $r<R$ is small enough. Thus we can bound (\ref{eqn:rem}) by
\[
C_{n,k,d} \,r^{d/2} \sum_{j=1}^{n(r)} \left( \int_{\sa} (u(\rho_j)+\delta)  (1-|\xi'|^2_{g_{_H}(x')})^{\frac{k-2}{2}} | \theta_j (1-\kappa_\delta)|^2  f u\, dm  \right)^{1/2}+C \delta^{1/2}.
\]
Since $\supp \theta_j|_{_\sa}\subseteq B_j$ we can use the bound $u(\rho_j)+\delta \leq u(\rho)+2\delta$. Continuing, we find
\begin{align*}
C_{n,k,d} & \,r^{d/2} \sum_{j=1}^{n(r)}  \left( \int_{\sa} (u(\rho_j)+\delta)  (1-|\xi'|^2_{g_{_H}(x')})^{\frac{k-2}{2}} | \theta_j (1-\kappa_\delta)|^2  f u\, dm  \right)^{1/2} \\
&\leq C_{n,k,d} \sum_{j=1}^{n(r)} m(B_j)^{1/2}  \left( \int_{\sa} (1-|\xi'|^2_{g_{_H}(x')})^{\frac{k-2}{2}} | \theta_j |^2  f u^2 \, dm  \right)^{1/2} + C_{n,k,d}  \,r^{d/2} n(r)^{1/2}  \left( \int_{\sa} \delta f u\, dm \right)^{1/2} \\
& \leq C_{n,k,d} \int_{\sa}  \sum_{j=1}^{n(r)} \left( \frac{1}{m(B_j)} \int_{B_j}  (1-|\xi'|^2_{g_{_H}(x')})^\frac{k-2}{2} f u^2   \,  dm \right)^{1/2}  \mathds{1}_{B_j} \, dm + C \delta^{1/2}.
\end{align*}
Therefore, combining the above steps we have
\begin{equation} \label{eqn:dct}
\limsup_{h\to 0^+} h^\frac{k-1}{2} |\langle \phi_h,\psi_h \rangle_{L^2(H)} | \leq  C \delta^{1/2} + C_{n,k,d} \int_{\sa}  \sum_{j=1}^{n(r)} \left( \frac{1}{m(B_j)} \int_{B_j}  (1-|\xi'|^2_{g_{_H}(x')})^\frac{k-2}{2} f u^2   \,  dm \right)^{1/2}  \mathds{1}_{B_j} \, dm,
\end{equation}
and since the left side does not depend on $r$, we may bound $\limsup_{h\to 0^+} h^\frac{k-1}{2} |\langle \phi_h,\psi_h \rangle_{L^2(H)} | $ by the limit of the right side of (\ref{eqn:dct}) as $r\to 0$. We will use the Dominated Convergence Theorem to bring the limit inside the integral. To simplify our computations, we will write
\[
F(\rho):= (1-|\xi'|^2_{g_{_H}(x')})^\frac{k-2}{2} f u^2.
\]
First we calculate the limit of the integrand in (\ref{eqn:dct}). Using the Lebesgue Differentiation Theorem \cite[Theorem 3.21]{Fol} and that each point in $\sa$ lies in finitely many balls of the cover, we see that
\[
\limsup_{r\to 0}   \sum_{j=1}^{n(r)} \left( \frac{1}{m(B_j)} \int_{B_j}  F   \,  dm \right)^{1/2}  \mathds{1}_{B_j}   \leq C_{n,k,d} \sqrt{F}= C_{n,k,d} |u| \sqrt{ (1-|\xi'|^2_{g_{_H}(x')})^\frac{k-2}{2} f } \qquad m-a.e.
\]
Lastly, to justify the use of the Dominated Convergence Theorem we need to show that the integrand in (\ref{eqn:dct}) is dominated by an $L^1$ function. We note that
\[
\sum_{j=1}^{n(r)} \left( \frac{1}{m(B_j)} \int_{B_j}  F   \,  dm \right)^{1/2}  \mathds{1}_{B_j}  \leq \sum_{j=1}^{n(r)} \sqrt{HF(\rho_j)}  \mathds{1}_{B_j}  \leq C \sqrt{HF(\rho)} \qquad m-a.e.
\]
where $H$ denotes the Hardy-Littlewood Maximal Functional. Furthermore, by the Maximal Theorem \cite[Theorem 3.17]{Fol} there exists a constant $C$ so that for all $t>0$ 
\[
m\left( \left\{ \rho \in\sa : HF(\rho) \geq t \right\} \right) \leq \frac{C}{t}
\]
which implies that  $\sqrt{HF}\in L^1(\sa,m)$. To see this we compute
\begin{align*}
\int_{\sa} \sqrt{HF(\rho)} \, dm = \int_\sa \int_0^{\sqrt{HF(\rho)}} \, dt \, dm &= \int_{\sa} \int \mathds{1}_{\{0\leq t\leq 1\}} \, dt \, dm + \int_{\sa} \int \mathds{1}_{\{1\leq t\leq \sqrt{HF(\rho)}\}} \, dt \, dm \\
&\leq m(\sa) + \int_1^\infty \int_{\sa}  \mathds{1}_{\{ \sqrt{HF(\rho)} \geq t  \}}  \, dm \, dt \\
&= C + \int_1^\infty m( \rho\in\sa : \{\sqrt{HF(\rho)} \geq t   \} ) \, dt \\
&\leq C + \int_1^\infty \frac{C}{t^2} \, dt < \infty,
\end{align*}
where we use the Fubini-Tonelli Theorem to change the order of integration in the second line.
Therefore, we are justified in applying the Dominated Convergence Theorem and we conclude that
\[
\limsup_{h\to 0^+} h^{\frac{k-1}{2}} \left| \langle \phi_h, \psi_h \rangle \right| \leq C\delta^{1/2} + C_{n,k,d} \int_{\sa} |u| \sqrt{ (1-|\xi'|^2_{g_{_H}(x')})^\frac{k-2}{2} f}\, dm\]
which holds for all $\delta>0$ and hence we obtain (\ref{eqn:nuleb}).
$\hfill\square$
%%%%%%%%%%%%%%%%%%%%%%%%%%%%%%%%%%%%%%%%%%%%%%%%%%%%%%%%%%%%%%%%%%%%%%%%%%%%%%%%%

\section{Localizing to $\sa$}\label{sec:ap} 
We first present two technical results which will be needed in the proof of Proposition \ref{prop}. First, Lemma \ref{lem:oh} tells us how to construct a cutoff $\tilde\chi\in\cci(T^*H)$ such that $\oph(\tilde\chi) \gh \oph(\chi)\phi_h$ is $O(h^\infty)$. Next, Lemma \ref{lem:ohinf} shows that the contributions of the inner product are negligible away from $\sa$. This section can be omitted on a first read.
Once again, throughout this section we assume $\{\phi_h\}$ is a compactly microlocalized collection of quasimodes on $M$ satisfying (\ref{eqn:quasi}) with defect measure $\mu$. We also assume that the sequence of functions $\{\psi_h\}$  on $H$ have defect measure $\nu$ and satisfy (\ref{eqn:psi}) and (\ref{eqn:psi2}).

%%%O H Infinity Lemma%%%%%%%%%%%%%%%%%%%%%%%%%%%%%%%%%%%%%%%%%%%%%%%%%%%%%%%%%%%%%%

The following lemma gives a condition for which the composition $\oph(\tilde\chi) \gh \oph(\chi)\phi_h$ is $O(h^\infty)$ where $\tilde\chi\in\cci(T^*H)$.
\begin{lemma} \label{lem:oh}
Let $\ct\in C^\infty (T^*H;[0,1])$ and $\chi\in\cci(T^*M;[0,1])$. Then 
\[
\oph(\ct)\gh \oph(\chi) \phi_h =O_{L^\infty(H)}(h^\infty)
\]
provided $\{\rho\in T^*_H M : \rho\in\supp\chi, \pi_{_{T^*H}}\rho\in\supp\ct\}=\emptyset$. 
\end{lemma}

\begin{proof}
We write $\oph(\ct)\gh \oph(\chi) \phi_h$ in coordinates:
\begin{multline*}
\oph(\ct)\gh \oph(\chi)\phi_h \\
=(2\pi h)^{k-2n} \iiint e^{\frac{i}{h} \langle x',\eta'  \rangle} e^{-\frac{i}{h} \langle y,\xi \rangle} \phi_h(y) \left( \int  e^{\frac{i}{h} \langle z',\xi'-\eta'  \rangle}  \ct(z',\eta') \chi(z',0,\xi',\xib)  dz' \right) dy\, d\xi \, d\eta'.
\end{multline*}
Consider the operator
\[
\mathcal{L}:=\frac{h\langle \xi'-\eta',D_{z'} \rangle}{|\xi'-\eta'|^2}
\]
which satisfies
\[
\mathcal{L}\, e^{\frac{i}{h} \langle z',\xi'-\eta'  \rangle} = e^{\frac{i}{h} \langle z',\xi'-\eta'  \rangle}. 
\]
We use $\mathcal{L}$ to repeatedly integrate by parts in the inner most integral. This is only possible provided $\xi'\not=\eta'$ on the support of $\ct \chi \big|_{\xb=0}$. However, we assumed that there are no points such that $(z',0,\xi',\xib)\in\supp \chi$ and $(z',\xi')\in\supp \ct$. Thus integrating by parts $N$ times using $\mathcal{L}$ in the $dz'$ integral we have
\begin{align*}
\Big| \int  e^{\frac{i}{h} \langle z',\xi'-\eta'  \rangle} & \ct(z',\eta') \chi(z',0,\xi',\xib)  dz' \Big| \\
 &=\left|  \int  \mathcal{L}^N e^{\frac{i}{h} \langle z',\xi'-\eta'  \rangle}  \ct(z',\eta') \chi(z',0,\xi',\xib)  dz' \right| \\
 &=  \left( \frac{h}{|\xi'-\eta'|^2} \right)^N \left| \int  e^{\frac{i}{h} \langle z',\xi'-\eta'  \rangle} \sum_{i_1,\dots,i_N} (\xi'_{i_1}-\eta'_{i_1})  \dots (\xi'_{i_N}-\eta'_{i_N}) D_{z'_{i_1},\dots,z'_{i_N}} (\ct \chi)  dz' \right| \\
 &\leq  \left( \frac{h}{|\xi'-\eta'|^2} \right)^N  \int C_N |\xi'-\eta'|^N \big|D_{z'}^N(\ct \chi)\big| dz' = C_N \left( \frac{h}{|\xi'-\eta'|} \right)^N  \int  \big|D_{z'}^N(\ct \chi)\big| dz'.   \\
\end{align*}
Furthermore we have
\begin{align*}
\Big| \oph(\ct)\gh & \oph(\chi) \phi_h  \Big| \\
& \leq C_N h^{N+k-2n} \iiint \frac{|\phi_h(y)|}{|\xi'-\eta'|^N}  \left( \int  \big|D_{z'}^N(\ct(z',\eta') \chi(z',0,\xi',\xib))\big| dz'\, \right) dy \,d\xi \, d\eta' \\
&=  C_N h^{N+k-2n} \|\phi_h\|_{L^1(M)} \iint f_N(\xi,\eta') |\xi'-\eta'|^{-N} d\xi \,d\eta' 
\end{align*}
where $f_N=\int  \big|D_{z'}^N(\ct(z',\eta') \chi(z',0,\xi',\xib))\big| dz'$ is smooth and compactly supported in $\xi$ since $\chi\in\cci(T^*M)$. Furthermore, since $\ct\chi$ is supported away from $\xi'=\eta'$ so is $f_N$. Also, since $\ct \chi$ is smooth and compactly supported in $z'$, we know the $dz'$ integral is finite. Moreover, for $N$ large enough $|\xi'-\eta'|^{-N}$ is highly localized in $\{|\xi'-\eta'|\leq 1\}$. The compactness in $\xi$ and this localization is enough to see that the last integral is finite and hence we have 
\[
\Big| \oph(\ct)\gh  \oph(\chi) \phi_h  \Big|  \leq C_{N,M} h^{N+k-2n}
\]
and hence $\oph(\ct)\gh \oph(\chi) \phi_h =O(h^\infty)$ as desired.
\end{proof}

Next, we show that away from $\sa$ the contributions from the generalized Fourier coefficients are negligible.
%%ORIGINAL O(h^infty) LEMMA%%%%%%%%%%%%%%%%%%%%%%%%%%%%%%%%%%%%%%%%%%%%%%%%%%%%%%%%%
\begin{lemma}\label{lem:ohinf}
Let $\chi_{_\ssm}\in\cci(T^*M)$ such that $\chi_{_\ssm}\equiv 1$ on a neighborhood of $\ssm$ and supported in a neighborhood of $\ssm$. Similarly let $\chi_{_A}\in\cci(T^*H)$ such that $\chi_{_A}\equiv 1$ on a neighborhood of $A$ and supported in a neighborhood of $A$. Then
\begin{equation}
  h^{\frac{k-1}{2}} \langle  \phi_h, \psi_h   \rangle_{L^2(H)} =  h^{\frac{k-1}{2}} \langle  \gh \oph(\chi_{_\ssm})\phi_h , \oph(\chi_{_A}) \psi_h \rangle_{L^2(H)} + o(1) \, \text{ as } \, h\to 0^+.
\end{equation}
\end{lemma}

\begin{proof}
First we use $\oph(\chi_{_\ssm}),\oph(1-\chi_{_\ssm}),\oph(\chi_{_A})$ and $\oph(1-\chi_{_A})$ to split up the inner product:
\begin{align}
 h^{\frac{k-1}{2}} \langle   \phi_h, &\psi_h   \rangle_{L^2(H)}  \nonumber \\
& =   h^{\frac{k-1}{2}}  \langle \gh  \oph(\chi_{_\ssm})\phi_h,  \psi_h   \rangle_{L^2(H)}  +  h^{\frac{k-1}{2}} \langle \gh \oph(1-\chi_{_\ssm})\phi_h, \psi_h  \rangle_{L^2(H)}  \nonumber  \\
& =   h^{\frac{k-1}{2}} \langle \gh  \oph(\chi_{_\ssm})\phi_h , \oph(\chi_{_A}) \psi_h \rangle_{L^2(H)}  +   h^{\frac{k-1}{2}} \langle  \gh \oph(\chi_{_\ssm})\phi_h  , \oph(1-\chi_{_A}) \psi_h\rangle_{L^2(H)} \nonumber \\
&\qquad +  h^{\frac{k-1}{2}}  \langle \gh  \oph(1-\chi_{_\ssm})\phi_h,  \psi_h   \rangle_{L^2(H)} \nonumber  \\
&=: I+I\!I+I\!I\!I. \label{eqn:III}
\end{align}
We just need to show that both $I\!I$ and $I\!I\!I$ are $o(1)$ as $h\to 0^+$. We begin with $I\!I\!I$. First, since $\phi_h$ is compactly microlocalized, there exists a cutoff $\chi\in \cci(T^*M)$ such that $\oph(1-\chi)\phi_h=O_{C^\infty}(h^\infty)$. Using $\chi$, we split $I\!I\!I$ once more,
\begin{align*}
I\!I\!I &=  h^{\frac{k-1}{2}}  \langle\gh \oph(1-\chi_{_\ssm}) \oph(\chi) \phi_h,  \psi_h   \rangle_{L^2(H)} + h^{\frac{k-1}{2}}  \langle \gh \oph(1-\chi_{_\ssm}) \oph(1-\chi) \phi_h,  \psi_h   \rangle_{L^2(H)} \\
& \leq h^{\frac{k-1}{2}}  \| \gh \oph(1-\chi_{_\ssm}) \oph(\chi) \phi_h \|_{L^2(H)} +  h^{\frac{k-1}{2}}  \| \gh \oph(1-\chi_{_\ssm}) \oph(1-\chi) \phi_h \|_{L^2(H)},
\end{align*}
where we also used that $\|\psi_h\|_{L^2(H)} =1$. Using that $\phi_h$ is compactly microlocalized, we observe that the term with $\oph(1-\chi)\phi_h$ is $O(h^\infty)$. Next, for the other term, we use an elliptic parametrix to rewrite
\[
\oph(1-\chi_{_\ssm})=R(h)\, P(h)+O(h^\infty)_{\Psi^{-\infty}}.
\]
To do this, we need verify that $\wfs(1-\chi_{_\ssm})\subseteq\ellp(P(h))$. Since $1-\chi_{_\ssm}$ does not depend on $h$, $\wfs(1-\chi_{_\ssm})=\esssupp(1-\chi_{_\ssm})\subseteq (\ssm)^c$. Moreover $\ellp(P(h))=\{p\not=0\}=(\ssm)^c$, and hence we have the inclusion necessary to use an elliptic parametrix. Therefore, we can write
\begin{align*}
h^{\frac{k-1}{2}} & \| \gh \oph(1-\chi_{_\ssm})  \oph(\chi) \phi_h\|_{L^2(H)}  \\
&= h^{\frac{k-1}{2}} \| \gh R(h) P(h) \oph(\chi) \phi_h\|_{L^2(H)}+ O(h^\infty) \\
& \leq  h^{\frac{k-1}{2}} \| \gh R(h)  \oph(\chi) P(h) \phi_h\|_{L^2(H)}+  h^{\frac{k-1}{2}} \| \gh R(h) \big(h \oph(H_p\chi)+O(h^2) \big) \phi_h\|_{L^2(H)}+  O(h^\infty) \\
& \leq C_k h^{-1/2}\| P(h) \phi_h\|_{L^2(M)} + C_k h^{1/2}  \|  \phi_h\|_{L^2(M)}+  O(h^\infty)
\end{align*}
where in the last line we used the standard restriction bound
\begin{equation}\label{eqn:stdrb}
\| \gh \oph(\kappa)u\|_{L^2(H)}\leq C_\gamma h^{-k/2}\| \oph(\kappa) u\|_{H_h^\gamma(M)} \leq C_k h^{-k/2} \|u\|_{L^2(M)}
\end{equation}
for $\gamma>k/2$, and $\kappa\in\cci(T^*M)$.   By (\ref{eqn:quasi}) we know $h^{-1}\|P(h)\phi_h\|_{L^2(M)}\to 0$ as $h\to 0^+$ and $\|\phi_h\|_L^2(M)=1$, and thus we obtain
\[
h^{\frac{k-1}{2}}  \| \gh \oph(1-\chi_{_\ssm})  \oph(\chi) \phi_h\|_{L^2(H)} =o(1) \, \text{ as } \, h\to 0^+
\]
as desired.

Next we show $I\!I$ is $O(h^\infty)$. To do so we first claim that there exists $\ct\in\cci(T^*H;[0,1])$ such that
\begin{equation}
\gh \oph(\chi_{_\ssm})\phi_h=\oph(\ct)\gh \oph(\chi_{_\ssm})\phi_h+O(h^\infty) \label{eqn:chit}.
\end{equation}
Using Lemma \ref{lem:oh} we find that we get (\ref{eqn:chit}) if we take $\ct(z',\xi')\equiv 1$ on a small neighborhood, $\mathcal{U}$, of $\{|\xi'|_{g_{_H}} \leq 1\}$ and supported in a small neighborhood of $\mathcal{U}$. Using  (\ref{eqn:chit}) we show $I\!I$ is $O(h^\infty)$. We rewrite
\begin{equation}\label{eqn:II}
 I\!I = \langle  \oph(\tilde{\chi}) \gh \oph(\chi_{_\ssm})\phi_h , \oph(1-\chi_{_A})\psi_h   \rangle_{L^2(H)} +  O(h^\infty) 
\end{equation}
Next observe
\begin{align*}
| \langle  \oph(\tilde{\chi}) & \gh \oph(\chi_{_\ssm})\phi_h , \oph(1-\chi_{_A})\psi_h   \rangle_{L^2(H)}  | \\
& \leq \| \gh \oph(\chi_{_\ssm})\phi_h \|_{L^2(H)} \|  \oph(\tilde{\chi})^*  \oph(1-\chi_{_A})\psi_h   \|_{L^2(H)}    \\
&\leq  C_k h^{-\frac{k}{2}}    \|  \oph(\tilde{\chi})^*  \oph(1-\chi_{_A})\psi_h   \|_{L^2(H)}
\end{align*}
where the last inequality follows from the standard restriction bound (\ref{eqn:stdrb}).
Recall $A=\wfs(\psi_h)$ and $\tilde\chi(x',\xi')$ is compactly supported in a neighborhood of $\{|\xi'|_{g_{_H}}\leq 1\}$. Let $K$ denote the support of $\tilde\chi$. There exists $\rho_j\in \overline{A^c\cap K}$ for $j=1,\dots N$ and $\theta_j\in\cci(T^*H;[0,1])$ supported sufficiently close to $\rho_j$ such that
\[
\|\oph(\theta_j)\psi_h\|_{L^2(H)}=O(h^\infty),
\]
and moreover 
\[
\Theta:=\sum_{j=1}^N \theta_j \equiv 1 \text{ on  } \overline{A^c\cap K}.
\]
We use an elliptic parametrix to rewrite
\[
\oph(\ct)^* \oph(1-\chi_{_A})=R(h) \, \oph(\Theta)+O(h^\infty)_{\Psi^{-\infty}}
\]
which we are allowed to do since $\wfs(\oph(\tilde\chi)^* \oph(1-\chi_{_A}))\subseteq \ellp(\Theta)$. To see this, note by properties of wavefront sets
\[
\wfs(\oph(\ct)^* \oph(1-\chi_{_A})) 
=\wfs(\oph(\ct)) \cap \wfs( \oph(1-\chi_{_A})) 
\subseteq K \cap A^c.
\]
Furthermore, $\ellp(\Theta) \supseteq \overline{A^c\cap K}$, and hence we have the inclusion needed to use the elliptic parametrix. Lastly, we have
\begin{align*}
\| \oph(\tilde\chi)^* \oph(1-\chi_{_A}) \psi_h \|_{L^2(H)}&= \| R(h)\, \oph(\Theta) \psi_h\|_{L^2(H)} +O(h^\infty) \\
& \leq  \| R(h) \|_{L^2\to L^2} \sum_{j=1}^N \| \oph(\theta_j) \psi_h \|_{L^2(H)}+O(h^\infty) =O(h^\infty).
\end{align*}
\end{proof}

%%%%%%%%%%%%%%%%%%%%%%%%%%%%%%%%%%%%%%%%%%%%%%%%%%%%%%%%%%%%%%%%%%

%%%%%%%%%%%%%%%%%%%%%%%%%%%%%%%%%%%%%%%%%%%%%%%%%%%%%%%%%%%%%%%%%%%%%%%%%%
\section{Localization to Geodesic Tubes: Proof of Proposition \ref{prop} } \label{sec:proofs}
%%%%%%%%%%%%%%%%%%%%%%%%%%%%%%%%%%%%%%%%%%%%%%%%%%%%%%%%%%%%%%%%%%%%%%%%%%%
In this section we finally present the proof of Proposition \ref{prop}.  Once again, throughout this section we assume $\{\phi_h\}$ is a compactly microlocalized collection of quasimodes on $M$ satisfying (\ref{eqn:quasi}) with defect measure $\mu$. We also assume that the sequence of functions $\{\psi_h\}$  on $H$ have defect measure $\nu$ and satisfy (\ref{eqn:psi}) and (\ref{eqn:psi2}). In the following we use coordinates $x=(x',\xb)$ such that $H=\{\xb=0\}$. Furthermore we write $\xb=(\xb_1,\xb_2,\dots,\xb_k)=(\xb_1,\xt)$.

We will need a few lemmas before proving the Proposition.

\subsection{A Technical Lemma}

\begin{lemma} \label{lem:13} Fix $\rho_0\in \sa$ and let $q\in\cci(\R_{\xb_1} \times \R_{\tilde{\xi}}^{k-1}  )$. 
There exists $T_0,R_0>0$ such that for all $0<T<T_0$ and $0<R<R_0$, if $\chi\in \cci(T^*M)$ is such that $\supp\chi\subseteq \T^{3T}(U)$ and $\supp H_p \chi \subset \T^{3T}(U)\setminus \T^{2T}(U)$, where $U\subseteq B_\mathscr{L}(\rho_0,R)$, then we have,
\begin{align*}
\|  \oph(q)  \oph(\chi) \phi_h(x',0,\xt)\|_{L^2_{x',\xt}} &\leq C\left( \frac{1}{\sqrt{T\hprho}} + \sqrt{2T} \right) \|\oph(\chi)\oph(q)\phi_h\|_{L^2_{x}} \\
&\qquad +  \frac{C \sqrt{2T}}{h} \Big( \|P\phi_h\|_{L^2_x}+\| \oph(\chi) [P,\oph(q)]\phi_h\|_{L^2_x}\Big)  +C_T h^{1/2} \| \phi_h \|_{L^2_x}.
\end{align*}
\end{lemma}
The proof of Lemma \ref{lem:13} is very similar to \cite[Lemma 13]{CG}, but we include it for completeness.

\begin{proof}Fix $\rho_0\in \sa$. Then, as before, we have $\partial_{\xib_1}p(\rho_0)>0$. Let $\mathcal{O}$ be an open neighborhood of $\rho_0$ such that  $\partial_{\xib_1}p>0$ on $\mathcal{O}$. Furthermore, let  $\T^{3T}(U)$ be a tube contained in $\mathcal{O}$. Then we can write
\[
p(x,\xi)=e(x,\xi) \big( \xib_1-a(x,\xi',\xit) \big) \, \, \text{ for } \, \, (x,\xi)\in\mathcal{O}
\]
where $e$ is elliptic on $\mathcal{O}$. Thus for $\ct\equiv 1$ on $\T^{3T}(U)$ and supported in $\mathcal{O}$ we have
\[
p(x,\xi) \ct(x,\xi)=e(x,\xi) \big( \xib_1-a(x,\xi',\xit) \big) \ct(x,\xi).
\]
Using the notation $P=\oph(p)$, observe
\begin{align*}
P\oph(\chi) &=P\oph(\ct)\oph(\chi)+O(h^\infty)\\
&= \left( \oph(p \ct) + h \oph(r_1) \right) \oph(\chi)+O(h^\infty) \\ 
&= \left(  \oph( e) \oph( (\xib_1 -a(x,\xi',\xit))  \ct ) + h \oph(r_2) +h \oph(r_1)\right) \oph(\chi) +O(h^\infty) \\ 
&= \oph(e) \Big( h D_{\xb_1} -\oph\big(a(x,\xi',\xit) \big) \Big) \oph(\chi)+h \oph(r)\oph(\chi) +O(h^\infty). 
\end{align*}
Thus
\[
\Big( h D_{\xb_1} -\oph\big(a(x,\xi',\xit) \big) \Big) \oph(\chi)\oph(q) \phi_h = \oph(e)^{-1} \Big( P\oph(\chi)\oph(q)-h\oph(r)\oph(\chi) \oph(q)  \Big) \phi_h
\]
where $\oph(e)^{-1}$ denotes a microlocal parametrix for $\oph(e)$ near $\supp\chi$. Since $a$ is a real symbol, we know that $\oph\big(a(x,\xi',\xit) \big)$ is an error of order $h$ away from being self adjoint. Therefore we can replace $\oph(a)$ with $\tilde{A}+h \tilde{R}$ where $\tilde{A}$ is self adjoint. Therefore we have
\[
\Big( h D_{\xb_1} -\tilde{A}\Big) \oph(\chi)\oph(q) \phi_h = \oph(e)^{-1} \Big( P\oph(\chi)\oph(q)-h\oph(r)\oph(\chi) \oph(q)  \Big) \phi_h +h\tilde{R} \oph(\chi)\oph(q) \phi_h
\]
We set 
\begin{align*}
u &:= \oph(\chi)\oph(q) \phi_h \\
f &:= \oph(e)^{-1} \Big( P\oph(\chi)\oph(q)-h\Big(\oph(r)-\oph(e)\tilde{R}   \Big)\oph(\chi) \oph(q)  \Big) \phi_h. 
\end{align*}
To later utilize the fact that $P\phi_h=o_{L^2(M)}(h)$ we rewrite $f$ as
\begin{align*}
f &= \oph(e)^{-1} \Big( \oph(\chi) \oph(q ) P  + [P,\oph(\chi)]\oph(q)   + \oph(\chi)[ P, \oph(q)]  \\
&\qquad -h \big( \oph(r)   -\oph(e)\tilde{R} \big) \oph(\chi)\oph(q)  \Big)\phi_h 
\end{align*}
Thus we have a differential equation for $u$:
\[
\Big( \partial_{\xb_1} - \frac{i}{h}\tilde{A}  \Big) u=\frac{i}{h}f
\]
To simplify notation, we write $\hat{x}$ to denote both $x'$ and $\tilde{x}$ and similarly $\hat{\xi}$ for $\xi',\xit$. First we define
\[
A(t,s,\hat{x}):=\int_t^s \tilde{A}(\xb_1,\hat{x},\hat{\xi}) d\xb_1
\]
We obtain
\begin{equation}\label{eqn:usol}
u(s,\xh)=e^{\frac{i}{h} A(t,s,\xh)}u(t,\xh)+\frac{i}{h}  \int_t^s  e^{-\frac{i}{h} A(s,\tau,\xh)}  f(\tau,\xh) d\tau.
\end{equation}
Next, define $\delta:=T \hprho$ and note for $T>0$
\[
0< \delta = T |\partial_{\xib}p(\rho_0)|= 2 T\sqrt{1-|\xi'_0|^2_{g_{_H}(x'_0)}} <2 T \qquad \text{where } \, \rho_0=(x_0',0,\xi_0',\xib_0)\in\sa.
\]
Further, let $\Phi(t)\in\cci(\R;[0,2\delta^{-1}])$ with $\supp \Phi\subseteq[0,\delta]$ and $\|\Phi\|_{L^1_t}=1$. Multiplying (\ref{eqn:usol}) through by $\Phi(t)$ and integrating in $t$ we have
\begin{align*}
u(s,\xh) &=\int_\R \Phi(t) u(s,\xh) dt  \\
&= \int_\R \Phi(t) e^{\frac{i}{h} A(t,s,\xh)}u(t,\xh) dt +\frac{i}{h} \int_\R \Phi(t) \int_t^s  e^{-\frac{i}{h} A(s,\tau,\xh)}  f(\tau,\xh) d\tau \, dt. 
\end{align*}
Next, taking the $L^2_{\xh}$ norm
\begin{align*}
\| u(s,\xh) \|_{L^2_{\xh}} & \leq \int_\R \Phi(t)  \left\| e^{\frac{i}{h} A(t,s,\xh)}u(t,\xh) \right\|_{L^2_{\xh}} dt +\frac{1}{h} \int_\R \Phi(t) \int_t^s  \left\|e^{-\frac{i}{h} A(s,\tau,\xh)}  f(\tau,\xh)  \right\|_{L^2_{\xh}} d\tau \, dt  \\
&= \int_\R \Phi(t)  \left\| u(t,\xh) \right\|_{L^2_{\xh}} dt +\frac{1}{h} \int_\R \Phi(t) \int_t^s  \left\|  f(\tau,\xh)  \right\|_{L^2_{\xh}} d\tau \, dt  
=: I + I\!I
\end{align*}
where the last line follows from
\[
\partial_s \left\| e^{\frac{i}{h} A(t,s,\xh)} u(t,\xh)  \right\|_{L^2_{\xh}}^2=2 \re \left\langle  \frac{i}{h} \tilde{A} e^{\frac{i}{h} A(t,s,\xh)} u(t,\xh) , e^{\frac{i}{h} A(t,s,\xh)} u(t,\xh)\right\rangle_{L^2_{\xh}}=0
\]
since $\tilde{A}$ is self adjoint. So $\| e^{\frac{i}{h} A(t,s,\xh)} \|_{L^2_{\xh}}=\| e^{\frac{i}{h} A(t,t,\xh)} \|_{L^2_{\xh}} = 1$.
Using H\"{o}lder's inequality and properties of $\Phi$ we bound $I$:
\[
I\leq \|\Phi\|_{L^2_t} \|u(t,\xh)\|_{L^2_{\xh,t}} \leq \frac{4}{\sqrt{\delta}} \|u(t,\xh)\|_{L^2_{\xh,t}} .
\]
To find a bound for $I\!I$, we first take the $L^\infty$ norm in $s$ and apply H\"{o}lder's inequality to get
\begin{equation}\label{eqn:IIf}
I\!I   \leq \frac{1}{h} \int \left\|  \mathds{1}_{[0,\delta]}(t) \mathds{1}_{[s,t]}(\tau) \right\|_{L^\infty_{t,s}} \|  f(\tau,\xh)\|_{L^2_{\xh}}  d \tau. 
\end{equation}
Splitting $f$ up into its components in (\ref{eqn:IIf}) we see that the first term is
\[
\frac{1}{h}  \int \left\|  \mathds{1}_{[0,\delta]}(t) \mathds{1}_{[s,t]}(\tau) \right\|_{L^\infty_{t,s}} \|  \oph(e)^{-1} \oph(\chi) \oph(q)P \phi_h(\tau,\xh)\|_{L^2_{\xh}}  d \tau 
\]
which is bounded by $C \sqrt{\delta} h^{-1}\| P \phi_h \|_{L_x^2}$.  We also have $\tau\leq t \leq \delta < 2T$, and recall that $\supp H_p\chi \equiv 0$ on $\{|\xb_1|\leq 2T\}$. Thus we can bound the second term by, 
\[
\| \oph(e)^{-1} [P,\oph(\chi)] \oph(q)\phi_h(\tau,\xh) \|_{L^2(\tau\in[-2T,2T],{\xh})} \leq C_T h^2 \| \phi_h \|_{L^2_{\xh}}.
\]
Continuing we obtain
\begin{align*}
I\!I 
&\leq \frac{C \sqrt{\delta}}{h} \Big( \|P\phi_h\|_{L^2_x}+ C_T h^2 \| \phi_h\|_{L^2_{\xh}} + \|\oph(e)^{-1} \oph(\chi) [P,\oph(q)]\phi_h\|_{L^2_x}  \\
&\qquad + h \| \oph(e)^{-1} \Big( \oph(r)-\oph(e)\tilde{R} \Big)\oph(\chi)\oph(q)\phi_h \|_{L^2_x} \Big) \\
& \leq \frac{C \sqrt{\delta}}{h}\|P \phi_h\|_{L^2_x}+ C_\delta \sqrt{\delta} h^{1/2} \| \phi_h\|_{L^2_{x}}+ \frac{C \sqrt{\delta}}{h} \| \oph(\chi) [P,\oph(q)]\phi_h\|_{L^2_x} + C \sqrt{\delta} \| \oph(\chi)\oph(q)\phi_h \|_{L^2_x} 
\end{align*}
where we used the standard estimate $\| \phi_h \|_{L^2_{\xh}} \leq C h^{-1/2} \| \phi_h \|_{L^2_x}$ in the last line.
So finally, combining the bounds for $I$ and $I\!I$ and rewriting $u$ as $\oph(\chi)\oph(q)\phi_h$ we have
\begin{align}
 \| \oph(\chi)  \oph(q)\phi_h(x',0,\xt)\|_{L^2_{x',\xt}} & \leq C\left( \frac{1}{\sqrt{T\hprho}} + \sqrt{2T} \right) \|\oph(\chi)\oph(q)\phi_h\|_{L^2_{x}} \nonumber \\
&  \qquad +  \frac{C \sqrt{2T}}{h} \Big( \|P\phi_h\|_{L^2_x}+ \| \oph(\chi) [P,\oph(q)]\phi_h\|_{L^2_x} \Big) + C_T  h^{1/2} \| \phi_h\|_{L^2_{x}}.  \label{eqn:chiq}
\end{align}
Therefore, using a commutator and the bound in (\ref{eqn:chiq}) we have
\begin{align*}
\| \oph(q)  \oph(\chi) \phi_h(x',0,\xt)\|_{L^2_{x',\xt}} & \leq  \|  \oph(\chi) \oph(q) \phi_h(x',0,\xt)\|_{L^2_{x',\xt}}  +  \| [ \oph(q), \oph(\chi) ]   \phi_h(x',0,\xt)\|_{L^2_{x',\xt}}  \\
&\leq C\left( \frac{1}{\sqrt{T\hprho}} + \sqrt{2T} \right) \|\oph(\chi)\oph(q)\phi_h\|_{L^2_{x}} \\
& \qquad  +  \frac{C \sqrt{2T}}{h} \Big( \|P\phi_h\|_{L^2_x}+\| \oph(\chi) [P,\oph(q)]\phi_h\|_{L^2_x}\Big)  +C_T h^{1/2} \| \phi_h \|_{L^2_x}, 
\end{align*}
where the estimate on the commutator term comes from the Sobolev embedding estimate:
\begin{align*}
  \| [ \oph(q), \oph(\chi) ]   \phi_h(x',0,\xt)\|_{L^2_{x',\xt}} & \leq h \| \oph(H_q \chi) \phi_h(x',0,\xt) \|_{L^2_{x',\xt}}  + O(h^2)\|\phi_h(x',0,\xt) \|_{L^2_{x',\xt}} \\
  & \leq C h \|  \phi_h(x',0,\xt) \|_{L^2_{x',\xt}}  \leq C h^{1/2} \| \phi_h \|_{L^2_x}
\end{align*}
which we regroup with the existing $O(h^{1/2})$ term.
\end{proof}

\subsection{Further localizing to Tubes}

The proof of Proposition \ref{prop} relies on decomposing $\supp \chi|_{_\sa}$ into many small "rectangles." Using the geodesic flow, we then extend the rectangles to create a collection of geodesic tubes covering $\supp \chi|_{_{\sa}}$. We get a much finer estimate on these tubes, which is given in the lemma below.

\begin{lemma} \label{lem} Fix $\rho_0=(x_0,0,\xi'_0,\xib_0)\in \sa$. There exist $T_0,R_0>0$ such that for all $0<T<T_0$ and $0<R<R_0$, if $U\subseteq \mathscr{L}$ is a neighborhood of $\rho_0$ contained in $B_\mathscr{L}(\rho_0,R)$, and $\chi\in \cci(T^*M)$ is such that $\supp\chi\subseteq \T^{3T}(U)$ and $\supp H_p \chi \subset \T^{3T}(U)\setminus \T^{2T}(U)$, then there exists a constant $C_k$ depending only on $k$ for which
\[
\limsup_{h\to 0^+} h^{k-1} \|\oph(\chi)\phi_h \|^2_{L^2(H)} \leq  \frac{C_k R^{k-1}}{ 2 T\sqrt{1-|\xi'_0|^2_{g_{_H}(x'_0)}}} \int_{T^*M} |\chi|^2 d\mu.
\]
\end{lemma}

To prove Lemma \ref{lem} we will strategically pick the $q$'s from Lemma \ref{lem:13} to be functions which vanish to order $k-1$ on the geodesic emanating from $\rho_0$. This will allow us to get a much better estimate on the geodesic tubes. 

\begin{proof} We choose local Fermi coordinates near $\rho_0\in\sa$ with respect to $H$, $(x',\bar{x})$ such that $H=\{\bar{x}=0\}$ and
\[
|\xi|^2_g=|\bar{\xi}|^2+f(x',\bar{x})|\xi'|^2.
\]
Thus note for $\rho_0$ we have $|\partial_{\bar{\xi}} p(\rho_0)|=2|\bar{\xi}|(\rho_0)>0 $ since $\bar{\xi}\in S^{k-1}_{\sqrt{1-|\xi'|^2_{g_{_H}}}}.$ We note the importance of the assumption that $A\Subset \{(x',\xi'):|\xi'|<1\}$ since otherwise we cannot assume $|\bar{\xi}|>0$ on $\sa$. Next, since $|\partial_{\bar{\xi}} p(\rho_0)|>0$ there exists a neighborhood $\mathcal{O}$ of $\rho_0$ such that $|\partial_{\bar{\xi}}p|>0$ on $\mathcal{O}$. Without loss of generality we assume that $\partial_{\xib_1}p(\rho_0)=|\partial_{\xib} p(\rho_0)|>0$ where $\bar{\xi}=(\bar{\xi}_1, \xib_2 \dots, \xib_{k})= (\bar{\xi}_1,\tilde{\xi})$. Furthermore, in these coordinates we have $\|u\|_{L^2_x}\leq 2 \|u\|_{L^2(M)}.$

\begin{figure}
\begin{center}
	\includegraphics[scale=0.25]{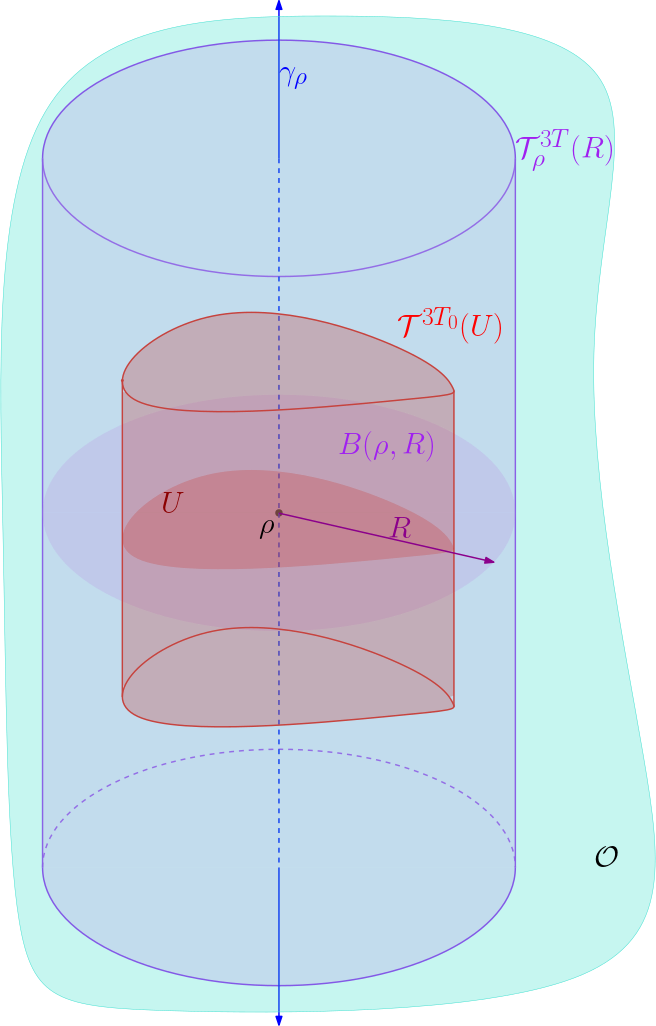}
	\end{center}
	\caption{}\label{fig:tubso}
\end{figure} 

Let $T_0,R_0>0$ be such that $\T_{\rho_0}^{3T_0}(R_0)\subset \mathcal{O}$, where the tube $\T$ is as defined in (\ref{eqn:tubes}). Note that for all $0<R<R_0$ we still have $\T^{3T_0}_{\rho_0}(R)\subseteq \mathcal{O}$. Therefore, the "flowout" time $T_0$ is independent of the tube width $R$, for $R<R_0$ small enough.   Let $\gamma_{\rho_0}(t)=(X(t),\Xi(t))$ denote a geodesic through $\rho_0$. From Hamilton's equations, we know the geodesic flow must satisfy
\[
\dot{\bar{x}}_1=p_{\bar{\xi}_1} \implies \frac{d}{dt} \bar{x}_1(t)=\partial_{\bar{\xi}_1}p(X(t),\Xi(t))=\partial_{\bar{\xi}_1} p(\gamma_{\rho_0}(t))>0 \, \text{  for  } |t|<3T_0 
\]
as $\partial_{\bar{\xi}_1} p(\rho)>0$ in $\mathcal{O}$. Thus for all $|t|\leq 3T_0$ we have   $\frac{d}{dt} \bar{x}_1(t)>0$. By the Inverse Function Theorem we can locally write $t=t(\bar{x}_1)$ and further we have
\[
\bar{x}_1 \big( \gamma_{\rho_o}(t(s)) \big)=s.
\]
We define
\begin{align*}
&X'(\xb_1):=x'\big(\gamma_{\rho_o}(t(\xb_1)) \big), \qquad \bar{X}_1(\xb_1):=\bar{x}_1\big( \gamma_{\rho_o}(t(\xb_1)) \big) =\xb_1     \\
&\tilde{X}(\xb_1):=\tilde{x} \big( \gamma_{\rho_o}(t(\xb_1)) \big), \qquad  \qquad \Xi(\xb_1):=\xi \big( \gamma_{\rho_o}(t(\xb_1)) \big).
\end{align*}
Therefore the geodesic through $\rho_0$ is parametrized by
\[
\xb_1 \mapsto (X'(\xb_1),\xb_1,\tilde{X}(\xb_1),\Xi(\xb_1))=(x' \big( \gamma_{\rho_o}(t(\xb_1)) \big), \xb_1, \tilde{x} \big( \gamma_{\rho_o}(t(\xb_1)) \big), \xi\big( \gamma_{\rho_o}(t(\xb_1)) \big) ).
\]
Moreover, we note that on the geodesic $\xi=\xi\big( \gamma_{\rho_0}(t(\xb_1))\big)=:\xi_0(\xb_1)$. This will be crucial in getting the improved estimate on the tubes. 

In what follows, we write $\tilde{x}$ to denote the normal coordinates to $H$ which are not $\bar{x}_1$, so $\bar{x}=(\xb_1,\xb_2,\dots \xb_k)=(\xb_1,\tilde{x})$. We first use a version of the Sobolev Embedding Theorem (see \cite[Lemma 6.1]{GalDM} or\cite[Corollary 8]{snap}):
\begin{align}
\|\oph(\chi) &\phi_h(x',\xb_1,\xt)\|_{L^\infty_{\xt}} \nonumber \\
& \leq C_k h^{\frac{1-k}{2}} \| \oph(\chi) \phi_h(x',\xb_1,\xt) \|^{1/2}_{L^2_{\xt}} \left(\sum_{i=2}^k \| (h D_{\xb_i}-\bar{\xi}_{0_i}(\xb_1))^{k-1} \oph(\chi) \phi_h(x',\xb_1,\xt)  \|^2_{L^2_{\xt}}  \right)^{1/4}. \nonumber
\end{align}
Squaring both sides, integrating with respect to $x'$ and applying H\"{o}lder's Inequality we have
\begin{align*}
\|\oph(\chi) &\phi_h(x',\xb_1,\xt) \|_{L^2_{x'}}^2  \leq \int \|\oph(\chi)\phi_h(x',\xb_1,\xt)\|^2_{L^\infty_{\xt}} dx' \\
&\leq C_k h^{1-k} \| \oph(\chi) \phi_h(x',\xb_1,\xt) \|_{L^2_{x',\xt}} \left( \sum_{i=2}^k \| (h D_{\xb_i}-\bar{\xi}_{0_i}(\xb_1))^{k-1} \oph(\chi) \phi_h(x',\xb_1,\xt) \|^2_{L^2_{x',\xt}} \right)^{1/2}.
\end{align*}
Setting $\xb_1=0$ and $\xt=0$ on the left we have
\begin{multline} \label{eqn:postse}
h^{k-1}\|\oph(\chi)\phi_h\|^2_{L^2(H)} \\
 \leq C_k \| \oph(\chi) \phi_h(x',0,\xt) \|_{L^2_{x',\xt}} \left( \sum_{i=2}^k \| (h D_{\xb_i}-\bar{\xi}_{0_i}(\xb_1))^{k-1} \oph(\chi) \phi_h(x',0,\xt) \|^2_{L^2_{x',\xt}} \right)^{1/2}.
\end{multline}
Next, we will use Lemma \ref{lem:13} to bound the $L^2$ norms on the right side of (\ref{eqn:postse}). We denote $T_{\rho_0}:=T \hprho = 2T\sqrt{1-|\xi'_0|^2_{g_{_H}(x'_0)}}$. For $q\in \cci(\R_{\xb_1}\times \R^{k-1}_{\xit})$ and $0<T<T_0$ we have 
\begin{align*} 
\| q(\xb_1,hD_{\xt}) \oph(\chi) & \phi_h(x',0,\xt) \|_{L^2_{x',\xt}}  \leq  C\left( \frac{1}{\sqrt{T_{\rho_0}}} + \sqrt{T} \right) \|\oph(\chi) q(\xb_1,hD_{\xt}) \phi_h\|_{L^2(M)}   \\
& \qquad   +  \frac{C \sqrt{T}}{h} \Big( \|P\phi_h\|_{L^2(M)} +  \| \oph(\chi) [P, q(\xb_1,hD_{\xt}) ]\phi_h\|_{L^2(M)} \Big)  +C_Th^{1/2}\|\phi_h\|_{L^2(M)}    
\end{align*}
where we have used that in our coordinates $\|u\|_{L^2_x}\leq 2 \|u\|_{L^2(M)}$. Next, since $P\phi_h=o_{L^2(M)}(h)$ we know $h^{-1}\|P\phi_h\|_{L^2(M)}\to 0$ as $h\to 0^+$. We also have that $C_T h^{1/2}\|\phi_h\|_{L^2(M)}\to 0$ as $h\to 0^+$ since $\|\phi\|_{L^2(M)}=1$. We regroup these two terms in a $o(1)$ error. Further, reordering the operators, we add an $O(h)$ error which we regroup with the $o(1)$ term to get
\begin{align}\label{eqn:q}
\| q(\xb_1,hD_{\xt}) \oph(\chi)  \phi_h(x',0,\xt) \|_{L^2_{x',\xt}} & \leq  C\left( \frac{1}{\sqrt{T_{\rho_0}}} + \sqrt{T} \right) \| q(\xb_1,hD_{\xt}) \oph(\chi) \phi_h\|_{L^2(M)}  \nonumber \\
& \qquad   +  \frac{C \sqrt{T}}{h}   \|  [P, q(\xb_1,hD_{\xt}) ] \oph(\chi) \phi_h\|_{L^2(M)}   + o(1)
\end{align}
First, taking $q=1$ in (\ref{eqn:q}) we get
\begin{equation}\label{eqn:q1}
\|  \oph(\chi) \phi_h(x',0,\xt) \|_{L^2_{x',\xt}} \leq C\left( \frac{1}{\sqrt{T_{\rho_0}}} + \sqrt{T} \right)  \|\oph(\chi) \phi_h \|_{L^2(M)} + o(1).
\end{equation}
Next, define $Q_i:=(h D_{\xb_i}-\bar{\xi}_{0_i}(\xb_1))^{k-1}=\oph(q_i)$ where $q_i=(\xib_i- \bar{\xi}_{0_i}(\xb_1))^{k-1}+O(h)$. Then, using this $q_i$ in (\ref{eqn:q}) we have
\begin{multline}\label{eqn:q2}
\| (h D_{\xb_i}-\bar{\xi}_{0_i}(\xb_1))^{k-1} \oph(\chi) \phi_h(x',0,\xt) \|_{L^2_{x',\xt}}  \\
\leq C\left( \frac{1}{\sqrt{T_{\rho_0}}} + \sqrt{T} \right) \| Q_i  \oph(\chi) \phi_h\|_{L^2(M)}  +  \frac{C \sqrt{T}}{h} \| [P, Q_i ]  \oph(\chi) \phi_h\|_{L^2(M)}  + o(1).
\end{multline}
Next define $\tilde{\chi}\in\cci(T^*M)$ such that $\supp\tilde{\chi}\subseteq \T^{3T}(U)$ and $\tilde{\chi}\equiv 1$ on $\supp\chi$.
We rewrite
\[
 Q_i  \oph(\chi)\phi_h=  Q_i   \oph(\tilde\chi)  \oph(\chi)  \phi_h + O_{}(h^\infty).
\]
Recall that on the geodesic $\gamma_{\rho_0}$ we have $\xi= \xi_{0}(\xb_1)$. Therefore, the principal symbol of $Q_i$, $\sigma(Q_i),$ vanishes to order $k-1$ on the geodesic. Furthermore, since $\tilde{\chi}$ is supported in the tube $\T^{3T}(U)$ where $U\subseteq B_\mathscr{L}(\rho_0,R)$, the distance between any point in $\supp\tilde\chi$ and the geodesic is approximately at most $R$. Thus we have
\[
\sup|\sigma(Q_i \oph(\tilde{\chi})  )|\leq 2 R^{k-1}.
\]
This implies that $\|Q_i \oph(\tilde\chi)  \|_{L^2\to L^2(M)}\leq 2 R^{k-1}+O(h)$ and in particular that
\begin{align*}
 \| Q_i \oph(\chi)   \phi_h \|_{L^2(M)} &=\| Q_i \oph(\tilde\chi)\oph(\chi)\phi_h \|_{L^2(M)} +O(h^\infty)\|\phi_h\|_{L^2(M)} \\
 &\leq (2 R^{k-1}+O(h))\|\oph(\chi)\phi_h\|_{L^2(M)}+O(h^\infty)\|\phi_h\|_{L^2(M)}.
\end{align*}
We also have that $H_p(\sigma(Q_i))=H_p( (\xib_i- \bar{\xi}_{0_i}(\xb_1))^{k-1} )=(k-1)( \xib_i- \bar{\xi}_{0_i}(\xb_1) )^{k-2}H_p( \xib_i- \bar{\xi}_{0_i}(\xb_1) )$ vanishes to order $k-1$ on the geodesic $\gamma_{\rho_0}$.  Since $\sigma([P,Q_i])=\frac{h}{i} H_p((\xib_i- \bar{\xi}_{0_i}(\xb_1))^{k-1}) $, we similarly have
\[
 \|[P,Q_i] \oph(\chi) \phi_h \|_{L^2(M)}\leq h (C_p R^{k-1}+O(h))\|\oph(\chi)\phi_h\|_{L^2(M)} +O(h^\infty)\|\phi_h\|_{L^2(M)}
\]
where $C_p$ is a constant which depends on $p$.
Using these estimates in (\ref{eqn:q2}) we have
\begin{multline}\label{eqn:qrest}
\| (h D_{\xb_i}-\bar{\xi}_{0_i}(\xb_1))^{k-1} \oph(\chi) \phi_h(x',0,\xt) \|_{L^2_{x',\xt}}  \\
\leq C\left( \frac{1}{\sqrt{T_{\rho_0}}} +\sqrt{T}\right) (2 R^{k-1}+O(h))\|\oph(\chi)\phi_h\|_{L^2(M)}  + C \sqrt{T} (C_p R^{k-1}+O(h))\|\oph(\chi)\phi_h\|_{L^2(M)}  +o(1).
\end{multline}
Finally, using (\ref{eqn:q1}) and (\ref{eqn:qrest}) in (\ref{eqn:postse}) and taking $h$ to zero, we have
\[
\limsup_{h\to 0^+} h^{k-1}  \|\oph(\chi)\phi_h\|^2_{L^2(H)} 
\leq \limsup_{h\to 0^+} C_k R^{k-1}  \left( \frac{1}{\sqrt{T_{\rho_0}}} + \sqrt{T} \right) \left( \frac{2}{\sqrt{T_{\rho_0}}}+C_p \sqrt{T}  \right) \|\oph(\chi) \phi_h \|^2_{L^2(M)}. 
\]
Using the defect measure $\mu$ associated to $\{\phi_h\}$ and that $T\ll 1$ we obtain the desired bound:
\[
\limsup_{h\to 0^+} h^{k-1}  \|\oph(\chi)\phi_h\|^2_{L^2(H)} \leq C_k \frac{R^{k-1}}{T_{\rho_0}}  \int_{T^*M} |\chi|^2 d\mu.
\]
\end{proof}
%%%%%%%%%%%%%%%%%%%%%%%%%%%%%%%%%%%%%%%%%%%%%%%%%%%%%%%%%%%%%
%%%%%%%%%%%%%%%%%%%%%%%%%%%%%%%%%%%%%%%%%%%%%%%%%%%%%%%%%%%

\subsection{Key Quantitative Estimate: Proof of Proposition \ref{prop}}
The main estimate used in the proof of Theorem \ref{thrm:6} lets us control terms of the form $|\langle \oph(\chi) \phi_h,\psi_h \rangle_{L^2(H)}|$. To prove it we first cover $\supp \chi|_{_\sa}$ with tubes and apply Lemma \ref{lem}. After localizing to the tubes, we will need to estimate $\langle \oph (\chi_j \chi) \phi_h, \psi_h \rangle_{L^2(H)}$ where $\chi_j$ is a cutoff localizing to a tube as in Lemma \ref{lem}. If we use Cauchy-Schwarz to bound this by the $L^2$ norms,
\[
| \langle \oph (\chi_j \chi) \phi_h, \psi_h \rangle_{L^2(H)} | \leq \| \oph(\chi_j \chi) \phi_h\|_{L^2(H)} \|\psi_h\|_{L^2(H)} = \| \oph(\chi_j \chi) \phi_h\|_{L^2(H)} 
\]
we lose all the information from $\psi_h$ since they are $L^2$-normalized on $H$. Thus we need to maintain the localization information of $\psi_h$ too. To do this, we will use $\chi_j \chi\in \cci (T^*M)$ and Lemma \ref{lem:oh} to find a new cutoff $\theta_j \in \cci(T^*H)$ such that
\[
\gh \oph (\chi_j \chi) \phi_h = \oph(\theta_j) \gh \oph (\chi_j \chi) \phi_h + O(h^\infty)
\]
and thus
\begin{align*}
 \langle \oph (\chi_j \chi) \phi_h, \psi_h \rangle_{L^2(H)} & =\langle \oph (\chi_j \chi) \phi_h,   \oph(\theta_j)^* \psi_h \rangle_{L^2(H)} + O(h^\infty) \\
 &\leq \| \oph (\chi_j \chi) \phi_h \|_{L^2(H)} \|  \oph(\theta_j)^* \psi_h\|_{L^2(H)}.
 \end{align*}
Then we will be able to apply Lemma \ref{lem} to the first term and use the defect measure for $\nu$ in the second.

%%%%%%%%%%%%%%%%%%%%%%%%%%%%%%%%%%%%%%%%%%%%%%%%%%%%%%%%%%%%%%%%%%%%%%%%%%%%%%%%%%%
%%%PROPOSITION 10 %%%%%%%%%%%%%%%%%%%%%%%%%%%%%%%%%%%%%%%%%%%%%%%%%%%%%%%%%%%

\begin{proof}[Proof of Proposition \ref{prop}]
Let $\chi\in\cci(T^*M)$ with $H_p \chi\equiv 0$ on $\Lambda^{2T}(\ep)$. 
Consider sets of the form:
\[
U_j:=\{(x',0,\xi',\xib)\in\sa :  (x',\xi')\in B_{A}(\rho_j',R_j), (x',0,\xi',\xib)\in B_{\sa} (\sigma_j(x',\xi'),R_j) \}
\]
where $R_j>0$, $\rho_j'=(x_j',\xi_j')\in A$ and $\sigma_j$ is a smooth section, that is $\sigma_j:A\to\sa$ and $\pi (\sigma_j(\rho'))=\rho'$. These are essentially "rectangles" in $\sa$ constructed by crossing a ball in $A$ with balls in the spheres $S^{k-1}$.  We note that the $\nu^A$ measure of these rectangles satisfy
\[
\nua(U_j) 
\geq C_{n,k} \int_{(x',\xi')\in A} \frac{\mathds{1}_{B_A(\rho_j',R_j)}}{\big(1-|\xi'|_{g_{_H}(x')}^2  \big)^{\frac{k-1}{2}}  } R_j^{k-1} d\nu(x',\xi')
\]
provided $R_j$ is small compared to $\sqrt{1-|\xi'|_{g_{_H}(x')}}$.
Furthermore by uniform continuity of $\log(1-|\xi'|^2_{g_{_H}(x')})$ on $\{|\xi'|^2_{g_{_H}(x')}<c<1\}$, there exists an $R>0$ independent of $\rho_j'$ such that if $(x',\xi')\in B_A(\rho_j',R)$ then
\[
\frac{1-k}{2} \log(1-|\xi_j'|_{g_{_H}(x_j')}^2)-\log(2) \leq \frac{1-k}{2} \log(1-|\xi'|_{g_{_H}(x')}^2) \leq \frac{1-k}{2} \log(1-|\xi_j'|_{g_{_H}(x_j')}^2)+\log(2).
\]
Thus for $R_j<R$, we have
\[
\nua(U_j) \geq  \frac{C_{n,k} \, \nu(B_A(\rho_j',R_j)) R_j^{k-1}}{ 2 (1-|\xi_j'|_{g_{_H}(x_j')}^2)^{\frac{k-1}{2} } } .
\]
Fix $\delta>0$. By outer regularity of $\nua$ there exist $\{U_j\}_{j=1}^{N(\delta)}$ covering $\supp\chi|_{_\sa}$ such that
\begin{equation}\label{eqn:cover}
\nua(\supp\chi|_{_\sa}) + \delta \geq  \sum_{j=1}^{N(\delta)} \nua(U_j)\geq C_{n,k} \sum_{j=1}^{N(\delta)}  \frac{  \nu(B_A(\rho_j',R_j))R_j^{k-1} }{ 2 (1-|\xi_j'|_{g_{_H}(x_j')}^2)^{\frac{k-1}{2} } }.
\end{equation}
To construct the cover of tubes, we first "thicken" the $U_j$'s into $U_j(\ep) \subseteq \mathscr{L}$ as defined in (\ref{eqn:thicc}). Finally, we flow $U^{\ep}_j$'s to form the collection of tubes $\{ \T^{3T}(U_j(\ep))\}$ where 
\begin{equation}\label{eqn:TR}
T\leq \min_j\{T_{0_j} \}=:T_0  \qquad \text{and} \qquad  3R_j\leq \min_j(R_{0_j})=:R_0  
\end{equation}
where the $T_{0_j}$'s are the "$T_0$'s" in the proof of lemma \ref{lem} and the $R_{0_j}$'s are  the "$R_0$'s" in the proof of lemma \ref{lem}. We note that $U_j(\ep)\subseteq B_\mathscr{L}(\rho_j,3R_j)$ where $\rho_j=\sigma_j(\rho'_j)$.
 By lemma \ref{lem:flow} (or  \cite[Lemma 3.5]{snap}), for each $j$, we can take $\chi_j\in\cci(T^*M;[0,1])$ supported in $\T^{3T}(U_j(\ep))$ such that $\supp H_p\chi_j \subseteq \T^{3T}(U_j(\ep)) \setminus \T^{2T}(U_j(\ep))$ and furthermore that 
\[
\sum_{j=1}^{N(\delta)} \chi_j\equiv 1 \text{ on } \bigcup_{|t|\leq 2T} \varphi_t ((\supp\chi|_{_\sa})(\ep/2)).
\]
Next we split the inner product into pieces localized to these tubes. We have
\begin{align*}
h^\frac{k-1}{2} \left| \left\langle \oph(\chi)\phi_h,\psi_h \right\rangle_{L^2(H)} \right| &\leq h^\frac{k-1}{2} \Big| \Big\langle \oph\Big(\sum_{j=1}^{N(\delta)} \chi_j \chi \Big)\phi_h,\psi_h \Big\rangle_{L^2(H)} \Big| +  h^\frac{k-1}{2}\Big| \Big\langle \oph\Big(\Big(1- \sum_{j=1}^{N(\delta)} \chi_j\Big) \chi \Big)\phi_h,\psi_h \Big\rangle_{L^2(H)} \Big| \\
&=:I+I\!I.
\end{align*}
We claim $I\!I=o(1)$ as $h\to 0^+$. We leave the proof of this to Lemma \ref{lem:ii} at the end of this section. The rest of this proof is dedicated to controlling $I$. By Lemma \ref{lem:oh} there exits $\theta_j\in\cci(T^*H)$ such that
\begin{equation} \label{eqn:thetaj}
\gh \oph(\chi_j \chi) \phi_h=\oph(\theta_j) \Gamma_H \oph(\chi_j\chi)\phi_h+O(h^\infty).
\end{equation}
Particularly, we need to take $\theta_j$ equal to  $1$ on $B_{T^*H}(\rho'_j,R_j+\ep)$ and $\supp\theta_j\subseteq B_{T^*H}(\rho'_j,R_j+2\ep)$.
Thus we have
\begin{multline}
I\leq h^\frac{k-1}{2} \sum_{j=1}^{N(\delta)} \left| \left\langle \oph\left( \chi_j \chi \right)\phi_h,\psi_h \right\rangle_{L^2(H)} \right| \leq  h^\frac{k-1}{2}\sum_{j=1}^{N(\delta)} \left| \left\langle \oph(\theta_j) \gh \oph\left( \chi_j \chi \right)\phi_h,\psi_h \right\rangle_{L^2(H)} \right| +O(h^\infty) \\
\leq h^\frac{k-1}{2} \sum_{j=1}^{N(\delta)} \| \oph\left( \chi_j \chi \right)\phi_h \|_{L^2(H)} \| \oph(\theta_j)^*  \psi_h\|_{L^2(H)} + O(h^\infty).
\end{multline}
We are now in position to apply Lemma \ref{lem} for "$\chi$"$=\chi_j \chi$.
\begin{align}
\limsup_{h\to 0^+} h^{\frac{k-1}{2}} & \left| \left\langle \oph(\chi)\phi_h,\psi_h \right\rangle_{L^2(H)} \right| = \limsup_{h\to 0^+} h^{\frac{k-1}{2}} (I+I\!I) \nonumber \\
&\leq \limsup_{h\to 0^+} h^{\frac{k-1}{2}}  \sum_{j=1}^{N(\delta)} \| \oph\left( \chi_j \chi \right)\phi_h \|_{L^2(H)} \| \oph(\theta_j)^*  \psi_h\|_{L^2(H)} \nonumber \\
&\leq   \sum_{j=1}^{N(\delta)} \left( \frac{C_k R_j^{k-1}}{ 2 T \sqrt{1-|\xi'_j|^2_{g_{_H}(x'_j)}}} \int_{T^*M} |\chi_j \chi|^2 d\mu \right)^{1/2} \left( \int_{T^*H} |\theta_j|^2 d \nu  \right)^{1/2} \label{eqn:3}
\end{align}
where we used that $\nu$ is a defect measure associated to $\{\psi_h\}$. Next, to get $\nua(\supp\chi|_{_\sa})$ to appear, we work to make the second term in (\ref{eqn:3}) look like the right side of (\ref{eqn:cover}). Moving the $R_j^{k-1}$ over, multiplying and dividing by $2 (1-|\xi_j'|^2_{g_{_H}(x_j')})^{\frac{k-1}{2}}$ and applying Cauchy-Schwarz we find that (\ref{eqn:3}) is bounded by
\[C_k \left(\frac{1}{T} \int_{T^*M}   \sum_{j=1}^{N(\delta)} \frac{  2 (1-|\xi_j'|^2_{g_{_H}(x_j')})^{\frac{k-1}{2}}  |\chi_j\chi|^2}{2 \sqrt{1-|\xi'_j|^2_{g_{_H}(x'_j)}}} d\mu   \right)^{1/2} \left(  \sum_{j=1}^{N(\delta)} \frac{R_j^{k-1}}{2 (1-|\xi_j'|^2_{g_{_H}(x_j')})^{\frac{k-1}{2}}} \int_{T^*H} \mathds{1}_{B_{T^*H}(\rho'_j,R_j+2\ep)} d \nu   \right)^{1/2}. \\
\]
Next, since the $\chi_j$'s are supported in the tubes, the first integral can be rewritten as an integral over ${\Lambda^{3T}}(\ep)$. Further, since the left side does not depend on $\ep$ we can take the limit as $\ep \to 0$ on the right side. Using the dominated convergence theorem to bring the limit inside we have
\begin{align*}
\limsup_{h\to 0^+} h^{\frac{k-1}{2}} & \left| \left\langle \oph(\chi)\phi_h,\psi_h \right\rangle_{L^2(H)} \right| \\
& \leq C_k \left(\frac{1}{T} \int_{\Lambda^{3T}}  |\chi |^2  \sum_{j=1}^{N(\delta)}   (1-|\xi_j'|^2_{g_{_H}(x_j')})^{\frac{k-2}{2}}  \left|\chi_j \right|^2 d\mu   \right)^{1/2} \left(  \sum_{j=1}^{N(\delta)} \frac{R_j^{k-1}  \nu\big( B_{A}(\rho'_j,R_j) \big)}{2 (1-|\xi_j'|^2_{g_{_H}(x_j')})^{\frac{k-1}{2}}}    \right)^{1/2} \\
\end{align*}
where $\Lambda^{3T}$ denotes $\T^{3T}(\sa)$. Next, since the second term is what we had in (\ref{eqn:cover}), we can replace it with $\left(  \nua(\supp\chi|_{_\sa})   +\delta \right)^{1/2}$. Noticing that the left side does not depend on $T$, we take the limit as $T\to 0$ and use the definition of $\mua$ from (\ref{eqn:mua}) to get
\begin{align*}
\limsup_{h\to 0^+} h^{\frac{k-1}{2}} & \left| \left\langle \oph(\chi)\phi_h,\psi_h \right\rangle_{L^2(H)} \right| \\
&\leq  C_{n,k} \left( \int_{\sa} |\chi|^2  \sum_{j=1} ^{N(\delta)} (1-|\xi_j'|^2_{g_{_H}(x_j')})^{\frac{k-2}{2}} \left|\chi_j \right|^2 d\mua \right)^{1/2}  \left(  \nua(\supp\chi|_{_\sa})   +\delta \right)^{1/2} .
\end{align*}
Finally, since $\{\chi_j|_{_\sa}\}$ formed a partition of unity for $\supp \chi|_{_\sa}$, $|\chi_j|\leq1$, $(1-|\xi'|^2_{g_{_H}(x')})^\frac{k-1}{2}  $ is continuous, and since $\delta>0$ was arbitrary, we have
\[
\limsup_{h\to 0^+} h^{\frac{k-1}{2}} \left| \left\langle \oph(\chi)\phi_h,\psi_h \right\rangle_{L^2(H)} \right| \leq C_{n,k} \left( \int_{\sa} |\chi|^2 (1-|\xi'|^2_{g_{_H}(x')})^\frac{k-2}{2}  d\mua \right)^{1/2}  \left(  \nua(\supp\chi|_{_\sa})  \right)^{1/2} 
\]
as desired.
\end{proof}

Finally, we show that term $I\! I=h^\frac{k-1}{2} \Big|  \Big\langle \oph \Big( \Big( 1-\sum_j \chi_j \Big) \chi \Big) \phi_h,\psi_h  \Big\rangle_{L^2(H)} \Big|$ in the proof of proposition \ref{prop} is $o(1)$ as $h\to 0^+$ as claimed above. 
%LEMMA TO GET II=O(h^infty)%%%%%%%%%%%%%%%%%%%%%%%%%%%%%%%%%%%%%%%%%%%%%%%%%%%%%%%%%%%
\begin{lemma}\label{lem:ii} For $\chi,\chi_j$ defined in the proof of proposition \ref{prop} we have
\[
h^\frac{k-1}{2} \Big|  \Big\langle \oph \Big( \Big( 1-\sum_j \chi_j \Big) \chi \Big) \phi_h,\psi_h  \Big\rangle_{L^2(H)} \Big|\to 0, \, \text{ as } \, h\to 0^+
\]
\end{lemma}

\begin{proof}
First, using Lemma \ref{lem:ohinf} we obtain
\begin{align*}
h^\frac{k-1}{2} & \Big|  \Big\langle \oph \Big( \Big( 1-\sum_j \chi_j \Big) \chi \Big) \phi_h,\psi_h  \Big\rangle_{L^2(H)} \Big|   \\
& \leq  h^\frac{k-1}{2} \Big| \Big\langle \oph\Big( \Big(1-\sum_{j=1}^{N(\delta)} \chi_j   \Big)  \chi \Big) \oph(\chi_{_{S^*M}})\phi_h, \oph(\chi_{_A}) \psi_h  \Big\rangle  \Big| +  o(1) \\
 & \leq h^\frac{k-1}{2} \left\| \oph(\chi_{_A})^* \gh  \oph\Big( \Big(1-\sum_{j=1}^{N(\delta)} \chi_j   \Big)  \chi \Big) \oph(\chi_{_{S^*M}})\phi_h  \right\|_{L^2(H)} \|\psi_h\|_{L^2(H)} + o(1)
\end{align*}
where $\chi_{_{S^*M}}$ and $\chi_{_A}$ are defined in the statement of Lemma \ref{lem:ohinf}. We show that
\begin{equation}\label{eqn:gross}
\oph(\chi_{_A})^* \gh  \oph\Big( \Big(1-\sum_{j=1}^{N(\delta)} \chi_j   \Big)  \chi \Big) \oph(\chi_{_{S^*M}})\phi_h =O(h^\infty).
\end{equation}
To do so, we employ Lemma \ref{lem:oh}. We just need to verify the hypothesis of the lemma. For contradiction, suppose there is a point $(z_0',\xi_0')\in\supp \chi_{_A}$ and also $(z_0',0,\xi_0',\xib_0)\in \supp \Bigl( \Big(1-\sum_j \chi_j \Big) \chi \chi_{_{S^*M}} \Bigr)$. First, we note that $(z_0,0,\xi_0',\xib_0)\not\in (\supp\chi\big|_{_\sa})(\ep/2)$. However, since also $(z'_0,0,\xi_0',\xib_0)\in \supp \chi_{_\ssm}$ and $(z_0',\xi_0')\in\supp\chi_{_A}$ we know that $(z'_0,0,\xi'_0,\xib_0)\in\sa({\alpha})=\bigcup_{\rho\in\sa} B_\mathscr{L}(\rho,\alpha)$ where $\alpha>0$ is small and depends on how tightly $\chi_{_\ssm}$ and $\chi_{_A}$ are localized. Furthermore $(z'_0,0,\xi'_0,\xib_0)\in\supp\chi$ and so we have
\begin{equation}\label{eqn:silly}
(z'_0,0,\xi_0',\xib_0)\not\in (\supp\chi\big|_{_\sa})(\ep/2) \qquad \text{and} \qquad (z'_0,0,\xi'_0,\xib_0)\in \supp\chi\big|_{_{{\sa}(\alpha)}}
\end{equation}
but by taking $\chi_{_A}$ and $\chi_{_\ssm}$ supported sufficiently close to $A$ and $\ssm$, we can find $\alpha$ such that $\supp\chi\big|_{_{\sa(\alpha)}}\subseteq  (\supp\chi\big|_{_\sa})(\ep/2) $ which contradicts (\ref{eqn:silly}). Thus use of Lemma \ref{lem:oh} is justified and we have (\ref{eqn:gross}).

\end{proof}

%%%%%%%%%%%%%%%%%%%%%%%%%%%%%%%%%%%%%%%%%%%%%%%%%%%%%%%%%%%%%%%%

\section{Recurrence: Proof of Theorem \ref{thrm:rec}}\label{sec:rec}
In this section we prove Theorem \ref{thrm:rec} which gives the behavior of $| \langle \phi_h, \psi_h \rangle_{L^2(H)}|$ as $h\to 0^+$ when the recurrent set of $\sa$ is $\nua$-measure zero.
First, we define the recurrent set and introduce some notation. Although the following can be defined more generally, we stick to defining loop set, recurrent set, etc., for $\sa$ only. First for each point $\rho\in\sa$ we define the first return time $T_A:\sa\to \R \cup\{ \infty\}$ by 
\[
T_A(\rho)=\inf\{t>0: \gamma_\rho(t) \in \sa\}
\]
where $\gamma_\rho(t)$ is the geodesic emanating from $\rho$. This gives us the first time in which the geodesic $\gamma_\rho(t)$ returns to $\sa$. If the geodesic never returns to $\sa$, the return time will be infinite. We will primarily be interested in the points which return to $\sa$ in finite time. We call the collection of such points the loop set, denoted
\[
\mathcal{L}_A=\{\rho\in\sa:T_A(\rho)<\infty\}.
\]
Since points in the loop set return to $\sa$ in finite time, we denote the point in which $\rho\in\mathcal{L}_A$ returns to by $\eta(\rho)$ defined by $\eta:\mathcal{L}_A\to \sa$,
\[\eta(\rho)=\gamma_\rho(T_A(\rho)).\]
Next, define the infinite loop sets
\[
\mathcal{L}_A^{+\infty}=\bigcap_{k\geq 0} \eta^{-k}(\mathcal{L}_A) \qquad \text{and} \qquad  \mathcal{L}_A^{-\infty}=\bigcap_{k\geq 0} \eta^{k}(\mathcal{L}_A) 
\]
which are essentially the loop set points that return to $\sa$ infinitely often forward and backward in time, respectively. Finally, the recurrent set $\mathcal{R}_A:=\mathcal{R}_A^+ \cap \mathcal{R}_A^-$ where
\[
\mathcal{R}_A^{\pm}:= \left\{  \rho\in \mathcal{L}_A^{\pm \infty} : \rho\in\bigcap_{N>0} \overline{ \bigcup_{k\geq N} \eta^{\pm k}(\rho) } \right\},
\]
which is essentially the collection of points $\rho\in\sa$ which return infinitely often and eventually get arbitrarily close to $\rho$.

\begin{proof}[Proof of Theorem \ref{thrm:rec}]
Suppose for contradiction that there is a sequence $h_j\to 0$ such that
\begin{equation}\label{eqn:cont}
|\langle \phi_{h_j}, \psi_{h_j} \rangle | \geq C h_j^{\frac{1-k}{2}}.
\end{equation}
Taking a subsequence if necessary, there exists defect measure $\mu$ for $\{\phi_{h_j}\}$. Further note that $\nu$ is still a defect measure for $\{\psi_{h_j}\}$. Defining $\mua$ as in (\ref{eqn:mua}) we decompose $\mua=f \nua+\lambda^A$. Then applying Theorem \ref{thrm:6} we have
\begin{align*}
\lim_{j\to\infty} h_j^{\frac{k-1}{2}} | \langle \phi_{h_j}, \psi_{h_j} \rangle | &\leq C_{n,k} \left( \int_{\sa} f (1-|\xi'|^2)^{\frac{k-2}{2}} d \nua \right)^{1/2} \\
& = C_{n,k} \left( \int_{\sa \cap \mathcal{R}_A} f (1-|\xi'|^2)^{\frac{k-2}{2}} d \nua   + \int_{\sa\setminus \mathcal{R}_A} f (1-|\xi'|^2)^{\frac{k-2}{2}} d \nua  \right)^{1/2}\\
&=  C_{n,k} \left(  \int_{\sa\setminus \mathcal{R}_A} f (1-|\xi'|^2)^{\frac{k-2}{2}} d \nua  \right)^{1/2} 
\end{align*}
where the last line follows from the fact that $\nua(\mathcal{R}_A)=0$. Next, since $\nua$ and $\lambda^A$ are mutually singular there exists $V$ and $W$ such that $\nua(V)=\lambda^A(W)=0$ and $\sa=V \sqcup W$. Therefore we have
\begin{align*}
\lim_{j\to\infty} h_j^{\frac{k-1}{2}} | \langle \phi_{h_j}, \psi_{h_j} \rangle | & \leq  C_{n,k} \left(  \int_{\sa\cap \mathcal{R}_A^c} f (1-|\xi'|^2)^{\frac{k-2}{2}} d \nua  \right)^{1/2}  \\
&= C_{n,k} \left(  \int_{(\sa\setminus \mathcal{R}_A) \cap V} f (1-|\xi'|^2)^{\frac{k-2}{2}} d \nua +  \int_{(\sa\setminus \mathcal{R}_A) \cap W} f (1-|\xi'|^2)^{\frac{k-2}{2}} d \nua  \right)^{1/2} \\
& = C_{n,k} \left(   \int_{(\sa\setminus \mathcal{R}_A) \cap W}  (1-|\xi'|^2)^{\frac{k-2}{2}} d \mua  \right)^{1/2} \leq C \mua(\sa\setminus \mathcal{R}_A)^{1/2}
\end{align*}
since $\nua(V)=0$ and since $\laa(W)=0$ on $W$, so we can rewrite $\mua=f\nua$ on $W$. Next, we use that Lemma \ref{lem:Ra} below gives $\mua(\sa\setminus \mathcal{R}_A)=0$. Thus
\[
\lim_{j\to\infty} h_j^{\frac{k-1}{2}} | \langle \phi_{h_j}, \psi_{h_j} \rangle | =0
\]
which contradicts (\ref{eqn:cont}).
\end{proof}

Finally, we show that $\sa\setminus\mathcal{R}_A$ is $\mua$-measure zero, which will complete the proof of Theorem \ref{thrm:rec}.

\begin{lemma}\label{lem:Ra}
Let $H\subseteq M$ and suppose that $\{\phi_h\}$ is a sequence of eigenfunctions with defect measure $\mu$. Then
\[
\mua(\mathcal{R}_A)=\mua(\sa).
\]
\end{lemma}
\begin{proof}
Let $B\subseteq \sa$ be an open set. For $\delta>0$ sufficiently small, define
\[
B_{\delta}=\bigcup_{|t|<3 \delta} \varphi_t (B).
\]
Since $(S^*M,\mu,\varphi_t)$ forms a measure preserving system, the Poincar\'{e} Recurrence Theorem implies that for $\mu$-a.e. $\rho\in B_{\delta}$ there exists $t_n^\pm \to \pm \infty$ such that $\varphi_{t_n^\pm}(\rho) \in B_\delta$. Moreover by definition of $B_\delta$ there exists $s_n^\pm$ such that $|s_n^\pm-t_n^\pm | <2 \delta$ and $\varphi_{s_n^\pm}(\rho)\in B\subseteq \sa$. Therefore, for $\mu$-a.e. $\rho\in B_\delta$ we have
\begin{equation}\label{eqn:prob}
\bigcap_{T>0} \overline{\bigcup_{t\geq T} \varphi_t(\rho)\cap B} \not=\emptyset \qquad \text{and} \qquad \bigcap_{T>0} \overline{\bigcup_{t\geq T} \varphi_{-t}(\rho)\cap B} \not=\emptyset
\end{equation}
since the sets, $\overline{\bigcup_{t\geq T} \varphi_{\pm t}(\rho)\cap B}$ are non-empty, compact, and nested as $T$ increases. 

Next, we show that (\ref{eqn:prob}) also holds for $\mua$-a.e. point in $B \subseteq \sa$. For contradiction, suppose that there is a set $B'\subseteq B$ with $\mua(B')>0$ and for each $\rho\in B'$, there exists a $T>0$ such that
\[
\bigcup_{t\geq T} \varphi_t(\rho)\cap B =\emptyset \qquad \text{or} \qquad \bigcup_{t\geq T} \varphi_{-t}(\rho)\cap B =\emptyset.
\]
Similarly to \cite[Lemma 6]{CGT}, we have $\mu|_{_{B_\delta}}=\mua dt$. Therefore, extending $B'$ to $B'_{\delta/3}=\bigcup_{|t|\leq \delta} \varphi_t (B')$ we have that
\[
\mu(B'_{\delta/3})=2 \delta \cdot \mua(B')>0.
\]
However, this implies (\ref{eqn:prob}) does not hold on $B'_{\delta/3} \subseteq B_\delta$ which is a set of positive $\mu$ measure, which is a contradiction.

Finally, let $\{B_k\}$ be a countable basis for topology on $\sa$. For all $k$ there exists a $B'_k \subseteq B_k$ of full $\mua$ measure such that for all $\rho\in B'_k$ (\ref{eqn:prob}) holds (with $B$ replaced with $B_k$). Let $X_k:= B'_k \cup (\sa \setminus B_k)$. Following the same argument as in \cite[Lemma 15]{CG} we find that $\bigcap_k X_k \subseteq \mathcal{R}_A$. However, we note $\mua(X_k)=\mua(B'_k)+\mua(\sa\setminus B_k)= \mua(B_k)+\mua(\sa\setminus B_k)=\mua(\sa)$. So each $X_k$ has full measure and thus $\bigcap_k X_k$ has full measure too. Therefore $\mathcal{R}_A$ has full measure too, and we have
\[
\mua(\mathcal{R}_A)=\mua(\sa)
\]
as desired.
\end{proof}

\vspace{.2in}
\bibliographystyle{alpha}
\bibliography{ricekrispies}

\begin{thebibliography}{Wym20b}

\bibitem[BGT07]{BGT}
N.~Burq, P.~G\'{e}rard, and N.~Tzvetkov.
\newblock Restrictions of the laplace-beltrami eigenfunctions to submanifolds.
\newblock {\em Duke Journal of Mathematics}, 138(3):445--486, 2007.

\bibitem[CG19a]{IEA}
Yaiza Canzani and Jeffrey Galkowski.
\newblock Improvements for eigenfunction averages: An application of geodesic
  beams.
\newblock arXiv:1809.06296, 2019.

\bibitem[CG19b]{CG}
Yaiza Canzani and Jeffrey Galkowski.
\newblock On the growth of eigenfunction averages: microlocalization and
  geometry.
\newblock {\em Duke Journal of Mathematics}, 168(16):2991--3055, 2019.

\bibitem[CG21]{ECVGB}
Yaiza Canzani and Jeffrey Galkowski.
\newblock Eigenfunction concentration via geodesic beams.
\newblock {\em Journal f{\"u}r die reine und angewandte Mathematik}, 2021.

\bibitem[CGT18]{CGT}
Yaiza Canzani, Jeffrey Galkowski, and John~A. Toth.
\newblock Averages of eigenfunctions over hypersurfaces.
\newblock {\em Communications in mathematical physics}, pages 619--637, 2018.

\bibitem[CS15]{CS}
Xuehua Chen and Christopher~D. Sogge.
\newblock On integrals of eigenfunctions over geodesics.
\newblock {\em Proceedings of the American Mathematical Society},
  143(1):151--161, 2015.

\bibitem[Fol99]{Fol}
Gerald~B. Folland.
\newblock {\em Real Analysis: Modern Techniques and Their Applications}.
\newblock Wiley, 2 edition, 1999.

\bibitem[Gal19a]{GalDM}
Jeffrey Galkowski.
\newblock Defect measures of eigenfunctions with maximal $ l^\infty $ growth.
\newblock {\em Annales de L'institut Fourier}, 69(4):1757--1798, 2019.

\bibitem[Gal19b]{snap}
Jeffrey Galkowski.
\newblock Minicourse on eigenfunctions-outline.
\newblock 2019.

\bibitem[Goo83]{good}
Anton Good.
\newblock {\em Local analysis of Selberg's trace formula}, volume 1040 of {\em
  Lecture Notes in Mathematics}.
\newblock Springer-Verlag, Berlin, 1983.

\bibitem[Hej82]{Hej}
Dennis~A. Hejhal.
\newblock Sur certaines s\'{e}ries de dirichlet associ\'{e}es aux
  g\'{e}od\'{e}siques ferm\'{e}es d'une surface riemann compacte.
\newblock {\em C.R. Acad. Sci. Paris S\'{e}r. I Math.}, 294(8):273--276, 1982.

\bibitem[STZ11]{STZ}
Christopher~D. Sogge, John~A. Toth, and Steve Zelditch.
\newblock About the blowup of quasimodes on riemannian manifolds.
\newblock {\em Journal of Geometric Analysis}, 21(1):150--173, 2011.

\bibitem[SXZ17]{SXZ}
Christopher~D. Sogge, Yakun Xi, and Cheng Zhang.
\newblock Geodesic period integrals of eigenfunctions on riemannian surfaces
  and the gauss-bonnet theorem.
\newblock {\em Cambridge Journal of Mathematics}, 5(1):123--151, 2017.

\bibitem[WXZ20]{WXZ}
Emmett~L. Wyman, Yakun Xi, and Steve Zelditch.
\newblock Fourier coefficients of restrictions of eigenfunctions.
\newblock arXiv:2011.11571, 2020.

\bibitem[WXZ21]{wxyz}
Emmett~L. Wyman, Yakun Xi, and Steve Zelditch.
\newblock Geodesic bi-angles and fourier coefficients of restrictions of
  eigenfunctions.
\newblock arXiv:2104.09470, 2021.

\bibitem[Wym17]{Wy17}
Emmett~L. Wyman.
\newblock Integrals of eigenfunctions over curves in surfaces of nonpositive
  curvature.
\newblock arXiv:1702.03552, 2017.

\bibitem[Wym19]{Wy19b}
Emmett~L. Wyman.
\newblock Looping directions and integrals of eigenfunctions over submanifolds.
\newblock {\em Journal of Geometric Analysis}, 29(2):1302--1319, 2019.

\bibitem[Wym20a]{Wy19a}
Emmett~L. Wyman.
\newblock Explicit bounds on integrals of eigenfunctions over curves in
  surfaces of nonpositive curvature.
\newblock {\em Journal of Geometric Analysis}, 30(3):3204--3232, 2020.

\bibitem[Wym20b]{Wy18}
Emmett~L. Wyman.
\newblock Period integrals in nonpositively curved manifolds.
\newblock {\em Mathematical Research Letters}, 27(5):1513--1664, 2020.

\bibitem[Xi17]{Xi}
Yakun Xi.
\newblock Inner product of eigenfunctions over curves and generalized periods
  for compact riemannian surfaces.
\newblock {\em Journal of Geometric Analysis}, 2017.

\bibitem[Zel92]{Zel}
Steven Zelditch.
\newblock Kuznecov sum formulae and szeg\"{o} limit formulae on manifolds.
\newblock {\em Communications in partial differential equations},
  17(1-2):221--260, 1992.

\bibitem[Zwo12]{zworski}
Maciej Zworski.
\newblock {\em Semiclassical Analysis}, volume 138 of {\em Graduate Studies in
  Mathematics}.
\newblock American Mathematical Society, Providence, RI, 2012.

\end{thebibliography}

\end{document}